\documentclass[12pt,reqno]{amsart}%
\usepackage{amsmath, amsfonts, amssymb, amsthm, amscd, amsbsy}
\usepackage{fancyhdr}
\usepackage[usenames,dvipsnames,svgnames,x11names,hyperref]{xcolor}
\usepackage{geometry}
\usepackage{graphicx}
\usepackage[pagebackref]{hyperref}%
\usepackage{amsmath}%
\usepackage{amsfonts}%
\usepackage{amssymb}
\usepackage{enumerate}

\usepackage[cmtip,all]{xy}
\usepackage{graphicx}
\usepackage{amssymb}
\usepackage{mathrsfs}
\usepackage{amsfonts}
\usepackage{amsmath}

\providecommand{\U}[1]{\protect\rule{.1in}{.1in}}
\hypersetup{backref=true,
pagebackref=true,
hyperindex=true,
colorlinks=true,
breaklinks=true,
urlcolor=NavyBlue,
linkcolor=Fuchsia,
bookmarks=true,
bookmarksopen=false,
filecolor=black,
citecolor=ForestGreen,
linkbordercolor=red
}
\newtheorem{theorem}{Theorem}[section]
\newtheorem*{theorem*}{Theorem}
\theoremstyle{plain}

\newtheorem{corollary}{Corollary}[section]

\newtheorem{lemma}{Lemma}[section]

\newtheorem{proposition}{Proposition}[section]
\newtheorem{remark}{Remark}[section]

\numberwithin{equation}{section}

\allowdisplaybreaks

\newcommand{\md}{\mathrm{d}}

\def\dim{\operatorname{dim}}

\def\R{\mathbb{R}}
\def\N{\mathbb{N}}

\def\R{\mathbb{R}}

\begin{document}

\title[Sharp fractional Sobolev and related inequalities ]{Sharp fractional Sobolev and related inequalities on  H-type groups }


\author{Yaojun Wang}

\address{School of Mathematical Sciences/Shanghai Center for Mathematical Sciences,
Fudan University,
220 Handan Road,
Shanghai,  200433,
People's Republic of China}
\email{yaojunwang22@m.fudan.edu.cn}

\author{Qiaohua  Yang}
\address{School of Mathematics and Statistics, Wuhan University, Wuhan, 430072, People's Republic of China}

\email{qhyang.math@whu.edu.cn}

\thanks{The second author  was partially supported by the National Natural
Science Foundation of China(No.11201346).}


%    General info
\subjclass[2000]{Primary: 46E35; 35R03;  22E25.}

%\date{January 1, 2001 and, in revised form, June 22, 2001.}


\keywords{Sobolev inequality; conformally invariant fractional powers; H-type group; best constant}

\begin{abstract}
We determine the sharp constants for  the fractional Sobolev inequalities associated with the   conformally invariant fractional powers $\mathcal{L}_{s}(0<s<1)$ of the sublaplacian on H-type groups. From these inequalities we derive
 a sharp log-Sobolev
inequality by  considering a limiting case and  a sharp Sobolev trace inequality.
The later extends  to this context   the  result of Frank,  Gonz\'alez,  Monticelli and Tan (Adv. Math, 2015) .

  \end{abstract}

\maketitle

%\section*{This is an unnumbered first-level section head}

\section{Introduction}

The classical Sobolev inequalities and their sharp constants have a wide range of applications  in analysis and geometry
 because they contain  geometric and probabilistic information. By using symmetrization arguments, Lieb  \cite{lieb1} (see also \cite{car, lieb2, fl3})
showed the  sharp Hardy-Littlewood-Sobolev inequality
\begin{align}\label{1.1}
\left|\int\int_{\mathbb{R}^{n}\times \mathbb{R}^{n}}\frac{f(x)\overline{g(y)}}{|x-y|^{\lambda}}\md x\md y\right|
\leq\pi^{\lambda/2}\frac{\Gamma((n-\lambda)/2)}{\Gamma(n-\lambda/2)} \left(\frac{\Gamma(n)}{\Gamma(n/2)}\right)^{1-\lambda/n}\|f\|_{p}\|g\|_{p},
\end{align}
where $0<\lambda<n$ and $p=\frac{2n}{2n-\lambda}$. The extremal functions in \eqref{1.1} are of the form
\begin{align*}
f(x)= g(x)=(1+|x|^{2})^{-(2n-\lambda)/2},
\end{align*}
up to the action of the conformal group of $\mathbb{R}^{n}$. By  a duality argument (see e.g. \cite{be1,be2, lieb2}), \eqref{1.1} is equivalent to the  sharp fractional Sobolev inequalities
\begin{align}\label{1.2}
\|(-\Delta)^{s/2}u\|^{2}\geq 2^{2s}\pi^{s}\frac{\Gamma((n+2s)/2)}{\Gamma((n-2s)/2)}
 \left(\frac{\Gamma(n/2)}{\Gamma(n)}\right)^{2s/n}\|u\|_{q}^{2},
\end{align}
where $0<s<n/2$ and $q=\frac{2n}{n-2s}$. Obviously,
for $s = 1$, \eqref{1.2} is the classical Sobolev inequality.

There is an  analogous Hardy-Littlewood-Sobolev  inequality on the Heisenberg group $\mathbb{H}^{n}$. Recall that $\mathbb{H}^{n}$ can be parameterized by
$(\mathbb{C}^{n}\times\mathbb{R},\circ)$ with the group law
\begin{align*}
(z,t)\circ (z',t')=(z+z',t+t'+2\textrm{Im}z\cdot z'),
\end{align*}
where $z\cdot z'=\sum\limits_{j=1}^{n}z_{j}\bar{z}_{j}'$.
The  homogeneous norm on $\mathbb{H}^{n}$ is given by
\begin{align*}
|(z,t)|=(|z|^{4}+t^{2})^{\frac{1}{4}}.
\end{align*}
The associated  homogeneous dimension is   $Q=2n+2$.
In a celebrated paper \cite{fl2}, Frank and Lieb determined the sharp constants
and extremal functions of Hardy-Littlewood-Sobolev  inequality on $\mathbb{H}^{n}$. We state the result  as follows:
\begin{theorem}[Frank-Lieb]\label{th1.1}
Let $0<\lambda<Q$ and $p=\frac{2Q}{2Q-\lambda}$. Then for any $f,g\in L^{p}(\mathbb{H}^{n})$,
\begin{align}\label{a1.3}
\left|\int\int_{\mathbb{H}^{n}\times \mathbb{H}^{n}}\frac{f(\xi)\overline{g(\eta)}}{|\xi^{-1}\circ \eta|^{\lambda}}\md \xi \md \eta\right|
\leq \left(\frac{\pi^{n+1}}{2^{n-1}n!}\right)^{\lambda/Q}\frac{n!\Gamma(\frac{Q-\lambda}{2})}{\Gamma^{2}(\frac{Q-2\lambda}{4})}\|f\|_{p}\|g\|_{p},
\end{align}
with equality if and only if, up to group translations and dilations,
\begin{align*}
f=c((1+|z|^{2})^{2}+t^{2})^{-\frac{2Q-\lambda}{4}},\;\;g=c'((1+|z|^{2})^{2}+t^{2})^{-\frac{2Q-\lambda}{4}}
\end{align*}
for some $c,c'\in\mathbb{C}$.
\end{theorem}
By  differentiating \eqref{a1.3}
 at the endpoints $\lambda = 0$ and $\lambda = Q$,  respectively, they  obtained a sharp logarithmic Hardy-Littlewood-Sobolev
inequality (see also \cite{br1}) and a sharp logarithmic Sobolev inequality.
In particular, choosing $\lambda=Q-2$ in Theorem \ref{th1.1} yields the Jerison-Lee inequality(see \cite{jl1,jl2,jl3}).

We remark that the  Frank-Lieb argument consists of two major steps. One is to show the   existence of extremal functions; the other is to evaluate the optimizer via
the second variation inequality with test functions provided by the moment zero condition.
Recently, Hang and Wang \cite{hang} presented a shorter proof of the Frank-Lieb inequality, in  which
they bypassed the subtle proof of existence via a subcritical approximation.

In the same paper, Frank and Lieb  \cite{fl2} conjectured  that their method could be applied to
 groups of Heisenberg type (in short, H-type groups) introduced by Kaplan \cite{k} (see  \cite{gv1} for the conjecture of the Jerison-Lee inequality on H-type groups).
In fact,  partial  results of Theorem \ref{th1.1}  have been generalized to the cases  of  quaternionic Heisenberg groups (see \cite{imv2,chz1}) and octonionic Heisenberg group (see \cite{chz2}) via Frank-Lieb approach.
We note these groups (i.e. Heisenberg groups,  quaternionic Heisenberg groups and octonionic Heisenberg group) are known as  the nilpotent component in the
Iwasawa decomposition of simple groups of rank one (in short, groups of Iwasawa type). They  are the subclasses of H-type groups. In fact,  most of H-type groups are not  groups of Iwasawa type (see  \cite{cdk}).

Very recently, the second author \cite{ya1} determined the sharp constant of the $L^2$ Folland-Stein inequality on any  H-type group, which extended the
Jerison-Lee inequality to this   context. The proof relied on combining these two methods, the Frank-Lieb approach and
the subcritical approximation method given by Hang and Wang \cite{hang}.

The aim of this paper is to look for the sharp constants of  fractional Sobolev inequalities associated with the   conformally invariant fractional powers $\mathcal{L}_{s}(0<s<1)$ of the sublaplacian on H-type groups. In addition to the two methods  above, the proof relies on
the ground state representation of $\mathcal{L}_{s}$, which has been studied  by Roncal and Thangavelu (see \cite{ron1,ron2}).
For more information about $\mathcal{L}_{s}$, we refer to Branson, Fontana, and Morpurgo \cite{br1} and Garofalo and Tralli \cite{ga1,ga2,ga3}.

To state our main result, we introduce some notations.
Recall that an H-type group $G$ is a
Carnot group of step two with the following properties (see Kaplan \cite{k}): the Lie
algebra $\mathfrak{g}$ of $G$ is endowed with an inner product $\langle \cdot,\cdot \rangle$
such that, if $\mathfrak{z}$ is the center of $\mathfrak{g}$, then
$[\mathfrak{z}^{\perp}, \mathfrak{z}^{\perp}]=\mathfrak{z}$.
Moreover, for every fixed $z\in \mathfrak{z}$, the map $J_{z}:
\mathfrak{z}^{\perp}\rightarrow \mathfrak{z}^{\perp}$ defined by
\begin{eqnarray*}%\label{1.2}
\langle J_{z}(v),\omega\rangle=\langle z,[v,\omega]\rangle,\;\;\forall \omega\in\mathfrak
z^{\perp},
\end{eqnarray*}
is an orthogonal map whenever $\langle z,z\rangle=1$. This implies that  the  dimension of $\mathfrak{z}^{\perp}$ is even.
For simplicity, we set
\begin{align*}
\dim \mathfrak{ z}^{\perp}=2n,\;\; \dim \mathfrak{z}=m.
\end{align*}
One can
 fix on $G$ a system of coordinates $(z,w)$ such that the group
law on $G$ has the form (see \cite{bu})
\begin{equation}\label{1.4}
\begin{split}
(z,w)\circ(z',w')=&\begin{pmatrix}
  z_{j}+z'_{j},\;\;j=1,2,\cdots, 2n \\
  w_{k}+w'_{k}+\frac{1}{2}\langle z,U^{(k)}z'\rangle,\;\;k=1,2,\cdots, m\\
\end{pmatrix},
\end{split}
\end{equation}
where the matrices $U^{(1)}, U^{(2)}, \cdots, U^{(m)}$ have the following
properties:

(1) $U^{(j)}$ is a $2n\times 2n$ skew symmetric and orthogonal matrix,
for every $j=1,2,\cdots, m$;

(2) $U^{(j)} U^{(k)}+U^{(k)}U^{(j)}=0$ for every $j,k\in
\{1,2,\cdots, m\}$ with $j\neq k$.

For $(z,w)\in G$, we denote by the  homogeneous norm of $(z,w)$
\begin{align*}
|(z,w)|=\left(\frac{|z|^{4}}{16}+|w|^{2}\right)^{\frac{1}{4}}.
\end{align*}
The associated  homogeneous dimension is
\begin{align*}
Q=2n+2m.
\end{align*}

Let $\mathcal{L}$ be the
sublaplacian on $G$ associated with an orthonormal basis of $\mathfrak{z}^{\perp}$ and
 $\nabla_{G}$ be the corresponding horizontal gradient.  Via the spectral formula,
 the   conformally invariant fractional powers $\mathcal{L}_{s}(0<s<1)$ can be defined by (see \cite{br1,ron1,ron2,ga1,ga2}),
 \begin{align}\label{1.5}
\mathcal{L}_{s}=2^{s}(-\Delta_{w})^{s}\frac{\Gamma\left(\frac{\mathcal{L}}{2(-\Delta_{w})^{1/2}}+\frac{1+s}{2}\right)}
{\Gamma\left(\frac{\mathcal{L}}{2(-\Delta_{w})^{1/2}}+\frac{1-s}{2}\right)},\;\; \Delta_{w}=\sum_{j=1}^{m}\frac{\partial^{2}}{\partial w_{j}^{2}}.
 \end{align}
The  Sobolev space $W^{s,2}(G)$ is defined  as the closure of the space of functions $u\in C_{0}^{\infty}(G)$
with respect to the norm
\begin{align*}
\|u\|_{W^{s,2}(G)}=\left(\int_{G}u(\xi)\mathcal{L}_{s}u(\xi)\md\xi\right)^{1/2}.
\end{align*}
To this end, we have
\begin{theorem}\label{th1.2}
It holds that, for $0<s<1$,
\begin{align}\label{1.6}
\int_{G}u(\xi)\mathcal{L}_{s}u(\xi)\md\xi\geq S_{n,m,s}\left(\int_{G}|u(\xi)|^{\frac{2Q}{Q-2s}}\md\xi\right)^{\frac{Q-2s}{Q}}, \; u\in W^{s,2}(G),
\end{align}
where
\begin{equation}\label{S_nms}
       S_{n,m,s}=4^{s\frac{Q+2n}{Q}}\pi^{s\frac{2n+m}{Q}}\frac{\Gamma(\frac{n+1+s}{2})\Gamma(\frac{n+m+s}{2})}{\Gamma(\frac{n+1-s}{2})\Gamma(\frac{n+m-s}{2})}\left(  \frac{\Gamma(n+\frac{m}{2})}{\Gamma (2n+m)}\right)^{2s/Q}.
\end{equation}
The inequality is sharp and an extremal function  is
\begin{align}\label{1.7}
U(z,w)=\left[\left(1+\frac{|z|^{2}}{4}\right)^{2}+|w|^{2}\right]^{-\frac{Q-2s}{4}}.
\end{align}
\end{theorem}

By an analogous  duality argument to the Euclidean case and using the following  Green's function of $\mathcal{L}_s$ (see  Lemma \ref{lm2.2}):
\begin{align}\label{c_nms}
\mathcal{L}_s^{-1}=c_{n,m,s}\frac{1}{|(z,w)|^{Q-2s}},\;\; c_{n,m,s}=\frac{\Gamma(\frac{n+1-s}{2})\Gamma(\frac{n+m-s}{2})}{2^{n+1+s} \pi^{n+\frac{m+1}{2}}\Gamma(s)},
\end{align}
we obtain the following
sharp Hardy-Littlewood-Sobolev inequality on $G$.
\begin{corollary}\label{co1.1}
Let $0<s<1$ and $p=\frac{2Q}{Q+2s}$. Then for any $f,g\in L^{p}(G)$,
\begin{align}\label{1.8}
\left|\int\int_{G\times G}\frac{f(\xi)\overline{g(\eta)}}{|\eta^{-1}\circ \xi|^{Q-2s}}\md \xi \md\eta\right|
\leq S_{n,m,s}^{-1}c_{n,m,s}^{-1}\|f\|_{p}\|g\|_{p},
\end{align}\label{1.3}
where $S_{n,m,s}$ and $c_{n,m,s}$ are given in \eqref{S_nms} and \eqref{c_nms}.
The inequality is sharp and the equality holds if
\begin{align*}
f(z,w)=g(z,w)=\left(\left(1+\frac{|z|^{2}}{4}\right)^{2}+|w|^{2}\right)^{-\frac{Q+2s}{4}}.
\end{align*}
\end{corollary}

\begin{remark}\label{remark1.1}
Let $$\mathbb{S}^{2n+m}=\{(z',w',t')\in\mathfrak{z^{\bot}}\oplus\mathfrak{z}\oplus \mathbb{R}:\;|z'|^{2}+|w'|^{2}+|t'|^{2}=1\}.$$
The Cayley transform $\mathscr{C}:\;G\rightarrow\mathbb{S}^{2n+m}$ is given
 by (see e.g. \cite{acb,cdk})
 \begin{align}\label{1.9}
\mathscr{C}(z,w)=\frac{1}{\mathscr{B}(z,w)}\left(\bar{\mathscr{A}}(z,w)z,2w,-1+\frac{|z|^{4}}{16}+|w|^{2}\right),
 \end{align}
where $\mathscr{A}(z,w)$ and $\bar{\mathscr{A}}(z,w)$ denote the
linear maps $1+\frac{|z|^{2}}{4}+J_{w}$ and
$1+\frac{|z|^{2}}{4}-J_{w}$ on $\mathfrak{z^{\bot}}$, and the real
number $\mathscr{B}(z,w)$ is defined by
\begin{align*}
\mathscr{B}(z,w)=\left(1+\frac{|z|^{2}}{4}\right)^{2}+|w|^{2}.
\end{align*}
Via the Cayley transform \eqref{1.9}, there is an equivalent version of \eqref{1.8} on $\mathbb{S}^{2n+m}$. However,
the expression of $|\eta^{-1}\circ \xi|$ on $\mathbb{S}^{2n+m}$ is  very cumbersome unless $G$ is a group of Iwasawa type (see \cite[Lemma 2.2]{acb}).

\end{remark}

By  differentiating
\eqref{1.8} at the endpoint $s = 0$, we obtain the following sharp logarithmic Sobolev inequality (see Beckner \cite{be3,be4} for logarithmic Sobolev inequality on $\mathbb{S}^{n}$).
\begin{theorem}\label{th1.3}
For any nonnegative $f$ with $\int_{G}f^{2}(\xi)\ln f^2(\xi)J_{\mathscr{C}}(\xi)\md\xi<\infty$ and
\begin{align}\label{b1.12}
\int_{G}f^{2}(\xi)J_{\mathscr{C}}(\xi)\md\xi=\int_GJ_{\mathscr{C}}(\xi)\md\xi =4^{n}\pi^{n+\frac{m}{2}}\frac{\Gamma(n+\frac{m}{2})}{\Gamma (2n+m)},
\end{align}
where
\begin{align}\label{1.10}
J_{\mathscr{C}}(z,w)=\left(\left(1+\frac{|z|^{2}}{4}\right)^{2}+|w|^{2}\right)^{-\frac{Q}{2}}
\end{align}
is the Jacobian determinant of Cayley transform \eqref{1.9},
one has
\begin{align*}%\label{log Sobolev}
   \int_G\int_G \frac{|f(\xi) -f(\eta)|^2}{|\eta^{-1}\circ\xi|^Q}J^{1/2}_{\mathscr{C}}(\xi)J^{1/2}_{\mathscr{C}}(\eta)\md \xi \md \eta
     \geq\frac{2^{n+3}\pi^{n+\frac{m+1}{2}}}{Q\Gamma(\frac{n+1}{2})\Gamma(\frac{n+m}{2})}\int_G f^2(\xi)\ln f^2(\xi) J_{\mathscr{C}}(\xi)\md \xi.
\end{align*}
Furthermore, the constant $\frac{2^{n+3}\pi^{n+\frac{m+1}{2}}}{Q\Gamma(\frac{n+1}{2})\Gamma(\frac{n+m}{2})}$ is sharp.
\end{theorem}

Finally, we obtain the following Sobolev trace inequality which
generalizes   the  result of Frank,  Gonz\'alez,  Monticelli and Tan (see \cite{fr1}).
\begin{theorem}\label{th1.4}
Let $0<s<1$ and  $\widetilde{W^{s,2}_{0}}(G\times [0,\infty))$ be the completion of $u\in C_{0}^{\infty}(G\times \mathbb{R})$
with respect to the norm
\begin{align*}
\left(\int_{0}^{\infty}\int_{G}\Big(|\partial_{\rho}u|^{2}+\frac{1}{4}\rho^{2}|\nabla_{w}u|^{2}+|\nabla_{G}u|^{2}\Big)\rho^{1-2s} \md z\md w\md\rho\right)^{1/2}.
\end{align*}
Then we have, for $u\in \widetilde{W^{s,2}_{0}}(G\times [0,\infty))$,
\begin{align}\nonumber
&\int_{0}^{\infty}\int_{G}\Big(|\partial_{\rho}u|^{2}+\frac{1}{4}\rho^{2}|\nabla_{w}u|^{2}+|\nabla_{G}u|^{2}\Big)\rho^{1-2s} \md z\md w\md\rho\\
\label{1.14}
\geq& 2^{1-2s}\frac{\Gamma(1-s)}{\Gamma(s)}S_{n,m,s}\left(  \int_G |u(z,w,0)|^\frac{2Q}{Q-2s}\md z\md w\right)^{\frac{Q-2s}{Q}},
\end{align}
where $S_{n,m,s}$ is given in \eqref{S_nms}. The inequality is sharp and an extremal function  is given by the Poisson integral of $U(z,w)$ (details of this integral  will be explained in  Section 5),
where $U(z,w)$ is given in \eqref{1.7}.
\end{theorem}

The organization of this paper is as follows. In Section 2, we will review some preliminary facts  that will be needed in the subsequent sections.
In Section 3, we  give another proof of  Hardy's  inequality via the  ground state representation of $\mathcal{L}_s$.
The
proof of Theorem \ref{th1.2} is  given in Section 4. In Section 5, we give the proofs of Theorems \ref{th1.3} and  \ref{th1.4}.

\section{Notations and preliminaries}

We begin by quoting some preliminary facts which will be needed in
the sequel.

\subsection{Sublaplacian on H-type groups}
In the rest of paper, we let $G$ be an H-type group with group law given by \eqref{1.4}.
The vector field in the Lie algebra
$\mathfrak{g}$  that agrees at the
origin with $\frac{\partial}{\partial z_{j}}(j=1,\cdots,2n)$ is given
by
\begin{equation*}
\begin{split}
X_{j} = \frac{\partial}{\partial
z_{j}}+\frac{1}{2}\sum^{m}_{k=1}\left(
\sum_{l=1}^{2n}U^{(k)}_{l,j}z_{l}\right)\frac{\partial}{\partial
w_{k}}
\end{split}
\end{equation*}
and $\mathfrak{g}$ is spanned by the left-invariant vector
fields
$$X_{1},\cdots,X_{2n},\frac{\partial}{\partial
w_{1}},\cdots,\frac{\partial}{\partial w_{m}}.$$
The horizontal gradient on $G$ is $\nabla_{G}=(X_{1},\cdots,X_{2n})$ and the sublaplacian on $G$ is given by
\begin{equation}\label{2.1}
\begin{split}
\mathcal{L}=-\sum^{2n}_{j=1}X^{2}_{j}&=-\sum^{2n}_{j=1}\left(\frac{\partial}{\partial
z_{j}}+\frac{1}{2}\sum^{m}_{k=1}\left(
\sum_{l=1}^{2n}U^{(k)}_{l,j}z_{l}\right)\frac{\partial}{\partial
w_{k}}\right)^{2}\\
&=-\Delta_{z}-\frac{1}{4}|z|^{2}\Delta_{w}-\sum^{m}_{k=1}\langle z,U^{(k)}\nabla_{z}\rangle\frac{\partial}{\partial
w_{k}},
\end{split}
\end{equation}
where
$$\Delta_{z}=\sum^{2n}_{j=1}\left(\frac{\partial}{\partial
z_{j}}\right)^{2},\;\;\Delta_{w}=\sum^{m}_{k=1}\left(\frac{\partial}{\partial
w_{k}}\right)^{2},\;\; \nabla_{z}=\left(\frac{\partial}{\partial
z_{1}},\cdots,\frac{\partial}{\partial z_{2n}}\right).$$
The nonisotropic dilation $\delta_{\mu}$ on $G$ is
\begin{align}\label{dilation}
 \delta_{\mu}(z,w)=(\mu z,\mu^{2}w),\;\;\mu>0.
\end{align}
It is easy to check, for $f\in C_{0}^{\infty}(G)$,
\begin{align*}
\mathcal{L}f(\delta_{\mu}(z,w))=&\mu^{2}(\mathcal{L}f)(\delta_{\mu}(z,w));\\
\Delta_{w}f(\delta_{\mu}(z,w))=&\mu^{4}(\Delta_{w}f)(\delta_{\mu}(z,w)).
\end{align*}
Therefore, by the definition of $\mathcal{L}_{s}$ (see \eqref{1.5}), we obtain
\begin{align}\label{2.2}
\mathcal{L}_{s}f(\delta_{\mu}(z,w))=&\mu^{2s}(\mathcal{L}_{s}f)(\delta_{\mu}(z,w)).
\end{align}

\subsection{Representation theory on H-type groups}

In this subsection we shall review  the representation theory on  H-type groups (see e.g. \cite{acbs}).
 We remark that it is similar to that on the Heisenberg group (see \cite{fo,tha1,tha2}).
 For more representation theory on nilpotent groups, we refer to \cite{cor,kiri}.

 Let $\mathfrak{z}^{\ast}$ be the dual of $\mathfrak z$.
  Fix $\lambda\in \mathfrak
z^{\ast}$ and define the skew symmetric
linear mapping $B(\lambda)$ on $\mathfrak z^{\perp}$ by
$$\langle B(\lambda)X,Y\rangle=\lambda([X,Y]),\;\;\;\;\forall X,Y\in\mathfrak z.$$
Let  $z_{\lambda}$ be the unique element of $\lambda\in \mathfrak
z^{\ast}$ such that
$$\langle B(\lambda)X,Y\rangle =\lambda([X,Y])=\langle J_{z_{\lambda}}(X),Y\rangle .$$
Denote the kernel of $B(\lambda)$ by $\mathfrak r_{\lambda}$ and let
$\mathfrak m_{\lambda}$ be the orthogonal complement of $\mathfrak
r_{\lambda}$ in $\mathfrak z^{\perp}$.  Denote by $\Lambda$
the Zariski-open subset of $\mathfrak z^{\ast}$ of the vectors
$\lambda$ for which $\dim\mathfrak m_{\lambda}$ is maximum.
Since $G$ is an H-type group, we have
$
\dim\mathfrak m_{\lambda}=\dim \mathfrak
z^{\perp}=2n,
$
 $\mathfrak r_{\lambda}=\{0\}$ for all $\lambda\in
\Lambda$ and $\Lambda=\mathfrak{z}^{\ast}\backslash \{0\}$.

For
$\lambda,\lambda'\in \mathfrak{z}^{\ast}\backslash \{0\}$, we put
\begin{align*}
\langle\lambda,\lambda'\rangle=\langle z_{\lambda},z_{\lambda'} \rangle,\;\; |\lambda|=\sqrt{\langle \lambda,\lambda\rangle}.
\end{align*}
We denote by
$\textrm{Sym}_{m}B(\lambda)$ the symmetric function of degree $2n$
in the roots of $B(\lambda)$.
Fixing $\lambda \in \mathfrak{z}^{\ast}\backslash \{0\}$, there are orthogonal vectors
$E_{1}(\lambda),\cdots,E_{n}(\lambda)$,
$\overline{E}_{1}(\lambda),\cdots,\overline{E}_{n}(\lambda)$ in
$\mathfrak m_{\lambda}$ such that
$$B(\lambda)E_{j}(\lambda)=|\lambda|\overline{E}_{j}(\lambda)\;\;\;\;\textrm{and}\;\;\;\;B(\lambda)\overline{E}_{j}(\lambda)=-|\lambda|E_{j}(\lambda),$$
and $\textrm{Sym}_{2n}B(\lambda)=|\lambda|^{2n}$.

Denote by
\begin{align*}
\mathfrak x_{\lambda}=\textrm{span}\{E_{1}(\lambda),\cdots,E_{n}(\lambda)\},\;\; \mathfrak y_{\lambda}=\textrm{span}\{\overline{E}_{1}(\lambda),\cdots,\overline{E}_{n}(\lambda)\}.
\end{align*}
Then we may write $V$ in $\mathfrak z^{\perp}$ as
$X+Y$ with $X\in\mathfrak x_{\lambda}$ and $Y\in\mathfrak
y_{\lambda}$.
The irreducible unitary representations parameterized by $\lambda\in
\mathfrak{z}^{\ast}\backslash \{0\}$ may be described as follows:
$$\left(\pi_{\lambda}(x,y,w)\phi\right)(y')=e^{i\sum^{m}_{j=1}\lambda_{j}w_{j}+i|\lambda|\sum^{n}_{j=1}(x_{j}\xi_{j}+\frac{1}{2}x_{j}y_{j})}\phi(y'+y),\;
x\in \mathfrak x_{\lambda}, \; y,y'\in \mathfrak y_{\lambda} $$
for all $\phi\in L^{2}(\mathfrak y_{\lambda})$. For $f\in L^{1}(G)$
and $\lambda\in \mathfrak{z}^{\ast}\backslash \{0\}$, its Fourier transform
$\widehat{f}(\lambda)$ is the operator-valued function defined by
$$\widehat{f}(\lambda)=\int_{\mathfrak z}\int_{\mathfrak y_{\lambda}}\int_{\mathfrak x_{\lambda}}f(x,y,w)\pi_{\lambda}(x,y,w)\md x\md y\md w,$$
which means that for each $\phi,\psi\in L^{2}(\mathfrak
y_{\lambda})$,
$$(\widehat{f}(\lambda)\phi,\psi)=\int_{\mathfrak z}\int_{\mathfrak y_{\lambda}}\int_{\mathfrak
x_{\lambda}}f(x,y,w)(\pi_{\lambda}(x,y,w)\phi,\psi)\md x\md y\md w.$$

The representation $\pi_{\lambda}$ of $G$ determines a
representation $\pi^{\ast}_{\lambda}$ of its Lie algebra $\mathfrak
g$ on the space of $C^{\infty}$ vectors. Recall that $\phi$ is said to
be a $C^{\infty}$ vector for the representation $\pi_{\lambda}$ if
$(x,y,w)\mapsto \pi_{\lambda}(x,y,w)\phi$ is a $C^{\infty}$ function from $G$
into the Hilbert space. The representation $\pi^{\ast}_{\lambda}$ is
defined by
$$\pi^{\ast}_{\lambda}(X)\phi=\left.\frac{\md}{\md t}\pi_{\lambda}(\exp tX)\phi\right|_{t=0}$$
for every $X$ in the Lie algebra $\mathfrak g$. We can then extend
$\pi^{\ast}_{\lambda}$ to the universal enveloping algebra of left-invariant differential operators on $G$. If $A$ is any such
operator, then a simple computation shows that
$$A(\pi_{\lambda}(x,y,w)\phi,\psi)=(\pi_{\lambda}(x,y,w)\pi^{\ast}_{\lambda}(A)\phi,\psi).$$
Therefore, one way of obtaining entry functions that are
eigenfunctions of $A$ is to take $\phi$ to be an eigenfunction of the  operator
$\pi^{\ast}_{\lambda}(A)$. One can calculate
$$H(\lambda)=\pi^{\ast}_{\lambda}(\mathcal{L})=-\sum^{n}_{j=1}\frac{\partial^{2}}{\partial\xi^{2}_{j}}+
|\lambda|^{2}|\xi|^{2},$$
which is nothing but  the Hermite operator in $\R^n$. The eigenfunctions
of $H(\lambda)$ are given by
$$\Phi^{\lambda}_{\alpha}(\xi)=|\lambda|^{\frac{n}{4}}\Phi_{\alpha}(\sqrt{|\lambda|}\xi),\;\;\alpha=(\alpha_{1},\cdots,\alpha_{n})\in \mathbb{N}^n,$$
 where $\Phi_{\alpha}(\xi)$ is the product
$\psi_{\alpha_{1}}(\xi_{1})\cdots \psi_{\alpha_{n}}(\xi_{n})$ and
$\psi_{\alpha_{j}}(\xi_{j})(j=1,\cdots,n)$ is the eigenfunction of
$- \frac{\partial^{2}}{\partial\xi^{2}_{j}}+\xi_{j}^{2}$ with
eigenvalue $2\alpha_{j}+1$ (see \cite{fo,tha1,tha2}).  Noticing  that
$$H(\lambda)\Phi^{\lambda}_{\alpha}=(2|\alpha|+n)|\lambda|\Phi^{\lambda}_{\alpha},\;\;|\alpha|=\alpha_{1}+\cdots+\alpha_{n},$$
we obtain
\begin{align}\label{2.3}
\mathcal{L}(\pi_{\lambda}(x,y,w)\Phi^{\lambda}_{\alpha},\Phi^{\lambda}_{\beta})=(2|\alpha|+n)|\lambda|(\pi_{\lambda}(x,y,w)\Phi^{\lambda}_{\alpha},\Phi^{\lambda}_{\beta}).
\end{align}
Thus the entry functions
$(\pi_{\lambda}(x,y,w)\Phi^{\lambda}_{\alpha},\Phi^{\lambda}_{\beta})$
as $\alpha,\beta$ ranging over $\mathbb{N}^{n}$ give a family of
eigenfunctions for the sublaplacian.

Let $P_k(\lambda)$ be the orthogonal projection defined by
\begin{align*}
    P_k(\lambda)\phi=\sum_{|\alpha|=k}(\phi,\Phi_{\alpha}^\lambda)\Phi_{\alpha}^\lambda,
\end{align*}
where $\alpha\in \N^n$ and $\phi\in L^2(\mathfrak y_{\lambda})$. By using  \eqref{2.3}, we have
\begin{align*}%\label{2.4}
(\mathcal{L}f)^{\widehat{}}=\widehat{f}(\lambda)\sum_{k=0}^\infty (n+2k)|\lambda|P_{k}(\lambda),\;\; f\in C_{0}^{\infty}(G).
\end{align*}
On the other hand, $(-\Delta_{w}f)^{\widehat{}}=|\lambda|^{2}\widehat{f}(\lambda)$.
Therefore,  the fractional powers $\mathcal{L}_{s}$ can be defined as follows:
\begin{align}\label{2.5}
    (\mathcal{L}_s f)^{\widehat{}}=\widehat{f}(\lambda)\sum_{k=0}^\infty (2|\lambda|)^{s}\frac{\Gamma (\frac{2k+n+1+s}{2})}{\Gamma(\frac{2k+n+1-s}{2})}P_{k}(\lambda),\;
    s\in \mathbb{R},\;f\in C_{0}^{\infty}(G),
\end{align}
provided the right side of \eqref{2.5} exists. Obviously,
\begin{align*}
\mathcal{L}_{s}^{-1}=\mathcal{L}_{-s}.
\end{align*}

\subsection{Green's function of $\mathcal {L}_{s}\;(0<s<n+1)$}
Let
\begin{align*}
    \varphi_{s,\rho}(z,w)=  \left( \left(\rho +\frac{|z|^2}4 \right)^2 +|w|^2  \right)^{-\frac{Q+2s}{4}},\;\; \rho>0.
\end{align*}
It has been shown by  Roncal and Thangavelu (see \cite[Theorem 3.7]{ron2})  (see also \cite{ga2} for $0<s<1$) that
\begin{align}\label{L_s varphi_-s rho}
    \mathcal{L}_s\varphi_{-s,\rho}(z,w)=(4\rho)^s\frac{\Gamma(\frac{n+1+s}{2})\Gamma(\frac{n+m+s}{2})}{\Gamma(\frac{n+1-s}{2})\Gamma(\frac{n+m-s}{2})} \varphi_{s,\rho}(z,w),\;\;0<s<n+1.
\end{align}
Furthermore, we have the following lemma (see \cite[Proposition 3.2]{ron1} for the case of Heisenberg group):
\begin{lemma}\label{lm2.1}
    Let $0<s<n+1$.  We have
    \begin{align}\label{2.6}
        \widehat{\varphi_{s,\rho}}(\lambda)=\sum_{k=0}^\infty c_{k,\rho}^\lambda(s)P_k(\lambda),
    \end{align}
    where
    \begin{equation}\label{ckrho}
       c_{k,\rho}^\lambda(s)= \frac{2^{n+1}\pi^{n+\frac{m+1}{2}}|\lambda|^s}{\Gamma(\frac{n+1+s}{2})\Gamma(\frac{n+m+s}{2})}
         L\left(\rho|\lambda|,\frac{2k+n+1+s}{2},\frac{2k+n+1-s}{2}\right)
    \end{equation}
  and the function $L(a,b,c)$ is given by
    \begin{align*}
          L(a,b,c)=\int_0^\infty e^{-a(2x+1)}x^{b-1}(1+x)^{-c}\md x.
    \end{align*}
\end{lemma}
\begin{proof}
 Since $\varphi_{s,\rho}(z,w)$ is radial in the $w$ variable,  we have
 \begin{align*}
      \widehat{\varphi_{s,\rho}}(\lambda)= \widehat{\varphi_{s,\rho}}(|\lambda|e_1),
 \end{align*}
where $e_1=(1,0,\cdots,0)\in \R^m$. Writing  $w=(w_1,w')\in \R\times \R^{m-1}$,  we have
\begin{align}\nonumber
   & \widehat{\varphi_{s,\rho}}(|\lambda|e_1)\\ \nonumber
   =&\int_{\R^{2n}}\int_{\R}\int_{\R^{m-1}} \left[ \left(\rho +\frac{|z|^2}{4}  \right)^2 +|w|^2  \right]^{-\frac{Q+2s}{4}}\pi_{|\lambda|e_1}(z,w)\md z\md w\\
   =&\int_{\R^{2n}}\int_{\R}\left\{\int_{\R^{m-1}} \left[ \left(\rho +\frac{|z|^2}{4}  \right)^2 +|w|^2  \right]^{-\frac{Q+2s}{4}}\md w'\right\}\pi_{|\lambda|}(z,w_1)\md z\md w_1,
   \label{2.7}
\end{align}
where $\pi_{|\lambda|}(z,w_1)=\pi_{|\lambda|}(x,y,w_1)$ is defined by
\begin{align*}
    (\pi_{|\lambda|}(x,y,w_1)\phi)(y')= e^{i |\lambda| w_1+i|\lambda|\sum_{j=1}^n(x_jy'_j+\frac{1}{2}x_jy_j)}\phi(y'+y),\; \phi\in L^2 (\R^n),
\end{align*}
 which is nothing but the irreducible unitary representation on $\mathbb{H}^{n}$.

We compute
\begin{align*}
&\int_{\R^{m-1}} \left[ \left(\rho +\frac{|z|^2}{4}  \right)^2 +|w|^2  \right]^{-\frac{Q+2s}{4}}\md w'\\
=&\left[ \left(\rho +\frac{|z|^2}{4}  \right)^2 +|w_1|^2  \right]^{-\frac{n+1+s}{2}}\int_{\R^{m-1}} \left(1 +|w'|^2  \right)^{-\frac{Q+2s}{4}}\md w'.
\end{align*}
By using (see \cite[(4.51)]{gv1})
\begin{align}\label{b2.16}
    \int_{\R^k} (1+|x|^2)^{-a}\md x=\pi^{k/2}\frac{\Gamma(a-k/2)}{\Gamma(a)}, \quad a>\frac{k}{2},
\end{align}
we obtain
\begin{align}\nonumber
&\int_{\R^{m-1}} \left[ \left(\rho +\frac{|z|^2}{4}  \right)^2 +|w|^2  \right]^{-\frac{Q+2s}{4}}\md w'\\
\label{2.8}
=&\pi^{\frac{m-1}{2}}\frac{\Gamma(\frac{n+1+s}{2})}{\Gamma(\frac{n+m+s}{2})}\left[ \left(\rho +\frac{|z|^2}{4}  \right)^2 +|w_1|^2  \right]^{-\frac{n+1+s}{2}}.
\end{align}
Substituting \eqref{2.8} into \eqref{2.7}, we obtain
\begin{align*}
\widehat{\varphi_{s,\rho}}(\lambda)=&
\pi^{\frac{m-1}{2}}\frac{\Gamma(\frac{n+1+s}{2})}{\Gamma(\frac{n+m+s}{2})}
\int_{\R^{2n}}\int_{\R}\left[ \left(\rho +\frac{|z|^2}{4}  \right)^2 +|w_1|^2  \right]^{-\frac{n+1+s}{2}}\pi_{|\lambda|}(z,w_1)\md z\md w_1\\
=&\sum_{k=0}^\infty c_{k,\rho}^\lambda(s)P_k(\lambda).
\end{align*}
To get the last equality, we use the fact
(see \cite[Proposition 3.2]{ron1})
\begin{align*}
    &\int_{\R^{2n}}\int_{\R} \left[\left(\rho+\frac{|z|^2}{4}  \right)^2+w_1^2 \right]^{-\frac{n+1+s}{2}}\pi_{|\lambda|}(z,w_1)\md z\md w_1\\
    =&\sum_{k=0}^\infty  \frac{(2\pi)^{n+1}|\lambda|^s}{\Gamma^2(\frac{n+1+s}{2})}L\left(\rho|\lambda|,\frac{2k+n+1+s}{2},\frac{2k+n+1-s}{2}\right) P_k(\lambda).
\end{align*}
This completes the proof of Lemma \ref{lm2.1}.
\end{proof}

Now we can compute the Green's function of $\mathcal {L}_{s}\; (0<s<n+1)$ (see \cite[(3.10)]{ron1} for the case of Heisenberg group and \cite[Theorem 1.2]{ga1} for $0<s<1$).
\begin{lemma}\label{lm2.2}
    Let $0<s<n+1$.  We have
    \begin{align*}
        \mathcal{L}_{s}^{-1}f=\mathcal{L}_{-s}f=c_{n,m,s}f\ast|(z,w)|^{-Q+2s},\;\; f\in C_{0}^{\infty}(G),
    \end{align*}
    in the sense of
 \begin{align*}%\label{L_{-s} f=}
   (\mathcal{L}_s^{-1}f)^{\widehat{}}(\lambda)= c_{n,m,s}(f*|(z,w)|^{-Q+2s})^{\widehat{}}(\lambda),
\end{align*}
   where $ c_{n,m,s}$ is given in \eqref{c_nms}
    and  $\ast$ is the convolution on  $G$ defined by
    \begin{align*}
    f\ast g(\xi)=\int_G f( \xi\circ \eta^{-1})g(\eta)\md\eta.
\end{align*}

\end{lemma}
\begin{proof}
By analytic continuation of $s$, \eqref{2.6} is also valid for $-n-1<s<0$. Therefore,
\begin{align*}
(f\ast\varphi_{-s,\rho})^{\widehat{}}(\lambda)=&\widehat{f}(\lambda) \widehat{\varphi_{-s,\rho}}(\lambda)\\
 =&\frac{2^{n+1+s}\pi^{n+\frac{m+1}{2}}(2|\lambda|)^{-s}}{\Gamma(\frac{n+1-s}{2})\Gamma(\frac{n+m-s}{2})}\times\\
&\widehat{f}(\lambda)\sum_{k=0}^\infty L\left(\rho|\lambda|,\frac{2k+n+1-s}{2},\frac{2k+n+1+s}{2}\right)P_k(\lambda).
\end{align*}
It is easy to verify that
\begin{align*}
  &e^{-\rho|\lambda|}\int_0^\infty (e^{-2\rho|\lambda| x}-1)x^{\frac{2k+n+1-s}{2}-1}(1+x)^{-\frac{2k+n+1+s}{2}}\md x \to 0, &\rho\to 0;\\
    &(e^{-\rho|\lambda|}-1)\int_0^\infty x^{\frac{2k+n+1-s}{2}-1}(1+x)^{-\frac{2k+n+1+s}{2}}\md x \to 0, &\rho\to 0
\end{align*}
  are both uniformly in $k$, so does
\begin{align*}
    & L\left(\rho|\lambda|,\frac{2k+n+1-s}{2},\frac{2k+n+1+s}{2}\right)\\
    =& \int_0^\infty e^{-\rho|\lambda|(2x+1)}x^{\frac{2k+n+1-s}{2}-1}(1+x)^{-\frac{2k+n+1+s}{2}}\md x\\
     \to & L\left(0,\frac{2k+n+1-s}{2},\frac{2k+n+1+s}{2}\right)\\
     =&\Gamma(s)\frac{\Gamma(\frac{2k+n+1-s}{2})}{\Gamma(\frac{2k+n+1+s}{2})}.
\end{align*}
Therefore, we have
\begin{align*}
   \lim_{\rho\to 0}(f\ast\varphi_{-s,\rho})^{\widehat{}}(\lambda)=
  \frac{2^{n+1+s}\pi^{n+\frac{m+1}{2}}\Gamma(s)}{\Gamma(\frac{n+1-s}{2})\Gamma(\frac{n+m-s}{2})}(\mathcal{L}_s^{-1}f)^{\widehat{}}(\lambda),
\end{align*}
which implies
\begin{align*}%\label{L_{-s} f=}
   (\mathcal{L}_s^{-1}f)^{\widehat{}}(\lambda)= \frac{\Gamma(\frac{n+1-s}{2})\Gamma(\frac{n+m-s}{2})}{2^{n+1+s} \pi^{n+\frac{m+1}{2}}\Gamma(s)}(f*|(z,w)|^{-Q+2s})^{\widehat{}}(\lambda).
\end{align*}
This completes the proof of Lemma \ref{lm2.2}.
\end{proof}
\section{Ground state representation of $\mathcal{L}_s$ and Hardy's inequalities}

In \cite[Lemma 5.1]{ron1}, Roncal and Thangavelu proved the following ground state representation of $\mathcal{L}_s$ on $\mathbb{H}^{n}$:
\begin{align}\label{a3.3}
\int_{\mathbb{H}^{n}} f(\xi)\mathcal{L}_s f(\xi)\md \xi
=C_{n,s}\int_{\mathbb{H}^{n}}\int_{\mathbb{H}^{n}} \frac{|f(\xi)-f(\eta)|^2}{|\eta^{-1}\circ\xi|^{Q+2s}}\md \xi \md \eta, f\in W^{s,2}(\mathbb{H}^{n}).
\end{align}
Using  the heat equation approach in \cite{ga1,ga2,ga3}, one sees that (\ref{a3.3}) can
be extended to any H-type group. For readers' convenience, we describe it as follows. It has been proved by Garofalo and Tralli (see \cite[Theorem 4.1]{ga3}) that the  heat kernel
with pole at the origin of the operator $\partial_{\rho\rho}+\frac{1-2s}{\rho}\partial_{\rho}+\frac{1}{4}\rho^2\Delta_w -\mathcal{L}-\partial_{t}$
for $0<s<1$
is given by
\begin{align*}
\mathfrak{q}_{(s)}((z,w),  t,\rho)=\frac{2^{m}}{(4\pi t)^{n+m+1-s}}\int_{\R^m}e^{-\frac{i}{t}\langle w ,\lambda\rangle} \left( \frac{|\lambda|}{\sinh |\lambda|} \right)^{n+1-s}e^{-\frac{|z|^2+\rho^{2}}{4t}\frac{|\lambda|}{\tanh |\lambda|}}\md \lambda.
\end{align*}
Set
\begin{align*}
K_s^t(z,w)=(4\pi t)^{1+s}\mathfrak{q}_{(-s)}((z,w),t, 0).
\end{align*}
By using (see \cite{ron2,ga2,ga3})
\begin{align*}
\mathcal{L}_s f(\xi)&=\frac{s}{\Gamma(1-s)} \int_0^\infty (f(\xi)-f*K_s^t(\xi))t^{-s-1}\md t
\end{align*}
and (see \cite[(1.17) and (1.19)]{ga2})
\begin{align*}
\int_0^\infty K_s^t (z,w) t^{-s-1}\md t=(4\pi)^{1+s} \int_{0}^{\infty}\mathfrak{q}_{(-s)}((z,w), t, 0)\md t= \frac{\Gamma(\frac{n+1+s}{2})\Gamma(\frac{n+m+s}{2})}{2^{n+1-s}\pi^{n+\frac{m+1}{2}}}|(z,w)|^{-Q-2s},
\end{align*}
one sees that
\begin{align}\nonumber
   \int_G f(\xi)\mathcal{L}_s f(\xi)\md \xi=&\frac{1}{2|\Gamma(-s)|}\int_{G}\int_0^\infty\int_G (f(\xi)-f(\eta)){f(\xi)} K_s^t(\eta^{-1}\circ\xi) t^{-s-1}\md t\md \xi\md \eta+\\
   \nonumber
  & \frac{1}{2|\Gamma(-s)|}\int_{G}\int_0^\infty\int_G (f(\eta)-f(\xi)){f(\eta)} K_s^t(\xi^{-1}\circ\eta) t^{-s-1}\md t \md \xi\md \eta\\
  \nonumber
   =&\frac{1}{2|\Gamma(-s)|}\int_G\int_G\int_0^\infty |f(\xi)-f(\eta)|^2 K_s^t(\eta^{-1}\circ\xi)t^{-s-1}\md t\md \xi\md \eta\\
   \label{representation L_s f,f}
    =&a_{n,m,s}\int_G\int_G \frac{|f(\xi)-f(\eta)|^2}{|\eta^{-1}\circ\xi|^{Q+2s}}\md \xi \md \eta,
\end{align}
    where
    \begin{align}\label{anms}
         a_{n,m,s}=2^{-n-2+s}\pi^{-n-\frac{m+1}{2}}\frac{\Gamma(\frac{n+1+s}{2})\Gamma(\frac{n+m+s}{2})}{|\Gamma(-s)|}.
    \end{align}
To get the second equality above, we use the fact    $K_s^t(\xi)=K_s^t(\xi^{-1})$.

By using the ground state representation technique given by  Frank, Lieb and Seiringer \cite{fl4} (see \cite[Theorem 5.2]{ron1} for the case of Heisenberg group), we can give another proof of
Hardy's  inequality associated with $\mathcal{L}_{s}$ (see  Roncal and Thangavelu  \cite[Corollary 1.4]{ron2}). We remark that
such technique also plays an important role in the proof of our main results (see Lemma \ref{L_s f w_i, fw_i}).
\begin{lemma}\label{lm3.2}
 Let $0<s<1$.    It holds that
    \begin{align}\label{Hardy-type U inequality}
   \int_G f\mathcal{L}_s f(z,w)\md z\md w  \geq    N_{n,m,s}\int_{G} \frac{|f(z,w)|^2}{[(1+\frac{|z|^2}{4})^2+|w|^2]^s}\md z\md w,\;\;\forall f\in W^{s,2}(G),
    \end{align}
    where
    \begin{equation}\label{N_nms}
        N_{n,m,s}=4^{s}
\frac{\Gamma(\frac{n+1+s}{2})\Gamma(\frac{n+m+s}{2})}{\Gamma(\frac{n+1-s}{2})\Gamma(\frac{n+m-s}{2})}
    \end{equation}
    and the equality holds if and only if
    $$f=c\left(\Big(1+\frac{|z|^2}{4}\Big)^2+|w|^2\right)^{-\frac{Q-2s}{4}},\;\; c\in\mathbb{R}.$$
\end{lemma}

\begin{proof}
Set, for $\rho>0$,
\begin{align*}
    \mathcal{H}_s[f]=\int_G f(\xi)\mathcal{L}_s f(\xi)\md \xi-(4\rho)^s\frac{\Gamma(\frac{n+1+s}{2})\Gamma(\frac{n+m+s}{2})}{\Gamma(\frac{n+1-s}{2})\Gamma(\frac{n+m-s}{2})}\int_{G} \frac{|f(z,w)|^2}{[(\rho+\frac{|z|^2}{4})^2+|w|^2]^s}\md z\md w.
\end{align*}
We first show
  \begin{align}\label{3.7}
        \mathcal{H}_s[f]=a_{n,m,s}\int_G\int_G \frac{|G(\xi)-G(\eta)|^2}{|\eta^{-1}\circ\xi|^{Q+2s}}\varphi_{-s,\rho}(\xi)\varphi_{-s,\rho}(\eta)\md\xi\md\eta,
\end{align}
    where $G(\xi)=f(\xi)\varphi_{-s,\rho}(\xi)^{-1}$ and  $a_{n,m,s}$ is given  in \eqref{anms}.
    Without loss of generality, we   assume $f\in C_{0}^{\infty}(G)$.

By polarizing \eqref{representation L_s f,f}, we have, for any $f,g\in W^{s,2}(G)$,
\begin{align}\label{formula for Hs}
   \int_G g\mathcal{L}_s f\md \xi=a_{n,m,s}\int_G\int_G \frac{(f(\xi)-f(\eta))(g(\xi)-g(\eta))}{|\eta^{-1}\circ\xi|^{Q+2s}}\md \xi \md\eta.
\end{align}
By applying \eqref{formula for Hs} to $g(\xi)=\varphi_{-s,\rho}(\xi)$ and $h(\xi)=|f(\xi)|^2g(\xi)^{-1}$, we get
\begin{align}\nonumber
   \int_G g\mathcal{L}_s h\md \xi&= a_{n,m,s}\int_G\int_G \frac{(h(\xi)-h(\eta))(g(\xi)-g(\eta))}{|\eta^{-1}\circ\xi|^{Q+2s}}\md \xi \md\eta\\
\nonumber
    &=a_{n,m,s}\int_G\int_G\left( |f(\xi)-f(\eta)|^2-\left|\frac{f(\xi)}{g(\xi)}-\frac{f(\eta)}{g(\eta)}   \right|^2 g(\xi)g(\eta)\right)\frac{\md \xi\md \eta}{|\eta^{-1}\circ\xi|^{Q+2s}}\\
    &=\int_G f\mathcal{L}_s f\md \xi-a_{n,m,s}\int_G\int_G\left|\frac{f(\xi)}{g(\xi)}-\frac{f(\eta)}{g(\eta)}   \right|^2 \frac{g(\xi)g(\eta)}{|\eta^{-1}\circ\xi|^{Q+2s}}\md \xi\md \eta. \label{3.9}
\end{align}
On the other hand, by using \eqref{L_s varphi_-s rho}, we have
\begin{align}\label{3.10}
    \int_G g\mathcal{L}_sh\md \xi=\int_G h\mathcal{L}_sg\md \xi&=(4\rho)^s\frac{\Gamma(\frac{n+1+s}{2})\Gamma(\frac{n+m+s}{2})}{\Gamma(\frac{n+1-s}{2})\Gamma(\frac{n+m-s}{2})} \int_G \frac{|f(\xi)|^2}{\varphi_{-s,\rho}(\xi)}\varphi_{s,\rho}(\xi)\md \xi.
\end{align}
Substituting \eqref{3.10} into \eqref{3.9} and using
\begin{align*}
    \frac{\varphi_{s,\rho}}{\varphi_{-s,\rho}}=  \left[\left(\rho+\frac{|z|^2}{4}\right)^2+|w|^2\right]^{-s},
\end{align*}
we obtain \eqref{3.7}.
 The desired result follows by using \eqref{3.7}.

\end{proof}

\section{Proof of Theorem \ref{th1.2}}

By an   interpolation of   Hardy's inequality \eqref{Hardy-type U inequality} and fractional Sobolev inequalities (see \cite[Theorem 2]{ka})
\begin{align}\label{4.1}
\int_{G}u(\xi)\mathcal{L}_{s}u(\xi)\md\xi\geq C\left(\int_{G}|u(\xi)|^{\frac{2Q}{Q-2s}}\md\xi\right)^{\frac{Q-2s}{Q}}, \; \forall  u\in W^{s,2}(G),
\end{align}
we obtain the following fractional Sobolev inequalities with weights.
\begin{proposition}\label{co4.1}
Let $0<s<1$ and  $2\leq p<\frac{2Q}{Q-2s}$. There exists $C>0$ such that for each $f\in W^{s,2}(G)$,
\begin{align}\label{4.4}
\int_{G}f\mathcal{L}_s f\md z\md w \geq C\left(\int_{G}|f|^{p}U(z,w)^{\frac{2Q}{Q-2s}-p}\md z\md w\right)^{\frac{2}{p}},
\end{align}
where $U(z,w)$ is the one given in \eqref{1.7}.
\end{proposition}

Next we shall show that the
embedding map
\begin{align}\label{a4.1}
W^{s,2}(G)\hookrightarrow L^{p}(G,
U(z,w)^{\frac{2Q}{Q-2s}-p}\md z\md w),\; 2\leq p<\frac{2Q}{Q-2s},
\end{align}
 is compact.

For simplicity, we define the ball centered at origin with radius $R$ by
\begin{align*}
 B_{R}(0)=\{(z,w)\in G: |(z,w)|<R\}.
\end{align*}
Set $\Sigma=\partial B_{1}(0)=\{(z,w)\in G: |(z,w)|=1\}$. We have the following polar coordinates on
$G$  (see \cite{fs3}):
\begin{align}\label{polar}
\int_{G}f(z,w)\md z\md w=
\int^{\infty}_{0}\int_{\Sigma}f(
\rho z^{\ast},\rho^{2}w^{\ast})\rho^{Q-1}\md\sigma \md\rho,\;\;f\in
L^{1}(G),
\end{align}
where  $z^{\ast}=\frac{z}{|(z,w)|}$ and $ w=\frac{w}{|(z,w)|^{2}}$.

Let $\phi : G\rightarrow [0,1]$ be a cut-off $C_{0}^{\infty}$ function
which  is equal to one in $B_{1}(0)$ and zero outside of $B_{2}(0)$.
For $f\in W^{s,2}(G)$, we let
\begin{align}\label{b4.5}
I_{R}(f)(\xi)=f(\xi)\phi\big(\delta_{\frac{1}{R}}(\xi)\big),\; \xi \in G.
\end{align}
By fractional Sobolev inequalities \eqref{4.1}, we have
\begin{align*}
I_{R}(f)\in L^{\frac{2Q}{Q-2s}}(B_{2R}(0))\subset L^{p}(B_{2R}(0)),\;\; 1\leq p<\frac{2Q}{Q-2s}.
\end{align*}

\begin{lemma}\label{lemma 4.2}
Let $0<s<1$ and $1\leq p<\frac{2Q}{Q-2s}$.
The embedding map
\[
I_{R}: W^{s,2}(G)\hookrightarrow
L^{p}(B_{2R}(0)),
\]
defined by \eqref{b4.5},  is compact.
\end{lemma}
\begin{proof}
We first show that there exists $C_1>0$ such that $\forall u\in W^{s,2}(G)$, it holds
\begin{equation}\label{part3}
    \int_{B_{2R}(0)}\int_{B_{2R}(0)} \frac{|I_{R}(u)(\xi)-I_{R}(u)(\eta)|^2}{|\eta^{-1}\circ\xi|^{Q+2s}}\md \xi \md \eta\leq C_1
    \int_{G}\int_{G} \frac{|u(\xi)-u(\eta)|^2}{|\eta^{-1}\circ\xi|^{Q+2s}}\md \xi \md \eta.
\end{equation}

In fact, we have
\begin{align}\nonumber
 &\int_{B_{2R}(0)}\int_{B_{2R}(0)} \frac{|I_{R}(u)(\xi)-I_{R}(u)(\eta)|^2}{|\eta^{-1}\circ\xi|^{Q+2s}}\md \xi \md \eta\\
 \nonumber
    =& \int_{B_{2R}(0)}\int_{B_{2R}(0)} \frac{|u(\xi)\phi\big(\delta_{\frac{1}{R}}(\xi)\big)-u(\eta)\phi\big(\delta_{\frac{1}{R}}(\eta)\big)|^2}{|\eta^{-1}\circ\xi|^{Q+2s}}\md \xi \md \eta\\
    \nonumber
    \leq& 2\left(\int_{B_{2R}(0)}\int_{B_{2R}(0)} \frac{|u(\xi)\phi\big(\delta_{\frac{1}{R}}(\xi)\big)-u(\eta)\phi\big(\delta_{\frac{1}{R}}(\xi)\big)|^2}{|\eta^{-1}\circ\xi|^{Q+2s}}\md \xi \md \eta\right.\\
    \nonumber
    &\quad+\left.\int_{B_{2R}(0)}\int_{B_{2R}(0)} \frac{|u(\eta)\phi\big(\delta_{\frac{1}{R}}(\xi)\big)-u(\eta)\phi\big(\delta_{\frac{1}{R}}(\eta)\big)|^2}{|\eta^{-1}\circ\xi|^{Q+2s}}\md \xi \md \eta\right)\\
   \nonumber
    \leq & 2\int_{G}\int_{G} \frac{|u(\xi)-u(\eta)|^2}{|\eta^{-1}\circ\xi|^{Q+2s}}\md \xi \md \eta+\\
    \nonumber
    &2\int_{B_{2R}(0)}\int_{B_{2R}(0)\cap \{|\eta^{-1}\circ\xi|>1\}} \frac{|u(\eta)|^2}{|\eta^{-1}\circ\xi|^{Q+2s}}\md \xi\md \eta
    +\\
     \label{part1}
    &2 \int_{B_{2R}(0)}\int_{B_{2R}(0)\cap \{|\eta^{-1}\circ\xi|\leq 1\}} \frac{|u(\eta)|^2|\phi\big(\delta_{\frac{1}{R}}(\xi)\big)-\phi\big(\delta_{\frac{1}{R}}(\eta)\big)|^2}{|\eta^{-1}\circ\xi|^{Q+2s}}\md \xi \md \eta.
\end{align}
 By  Taylor formula on $G$ (see \cite[Theorem 2]{bo}), there exists $C_2>0$ such that
 \begin{align}\label{b4.8}
\left|\phi\big(\delta_{\frac{1}{R}}(\xi)\big)-\phi\big(\delta_{\frac{1}{R}}(\eta)\big)\right|\leq C_2 |\eta^{-1}\circ\xi|,\; \forall \xi, \eta \in G.
 \end{align}
Substituting \eqref{b4.8} into \eqref{part1}, we obtain
\begin{align*}
&\int_{B_{2R}(0)}\int_{B_{2R}(0)} \frac{|I_{R}(u)(\xi)-I_{R}(u)(\eta)|^2}{|\eta^{-1}\circ\xi|^{Q+2s}}\md \xi \md \eta\\
    \leq & 2\int_{G}\int_{G} \frac{|u(\xi)-u(\eta)|^2}{|\eta^{-1}\circ\xi|^{Q+2s}}\md \xi \md \eta+2\int_{B_{2R}(0)}\int_{B_{2R}(0)\cap \{|\eta^{-1}\circ\xi|>1\}} \frac{|u(\eta)|^2}{|\eta^{-1}\circ\xi|^{Q+2s}}\md \xi\md \eta+\\
    &2C_2 \int_{B_{2R}(0)}\int_{B_{2R}(0)\cap \{|\eta^{-1}\circ\xi|\leq 1\}} \frac{|u(\eta)|^2}{|\eta^{-1}\circ\xi|^{Q+2s-2}}\md \xi\md \eta.
\end{align*}
By  polar coordinates \eqref{polar}, we have
\begin{align*}
\int_{B_{2R}(0)\cap \{|\eta^{-1}\circ\xi|\leq 1\}} \frac{1}{|\eta^{-1}\circ\xi|^{Q+2s-2}}\md \xi\leq& C_3,\\
\int_{B_{2R}(0)\cap \{|\eta^{-1}\circ\xi|>1\}} \frac{1}{|\eta^{-1}\circ\xi|^{Q+2s}}\md \xi\leq& C_4,
\end{align*}
where  $C_3$ and $C_4$  are positive constants independent of $\eta$. Therefore,
\begin{align}\nonumber
 &\int_{B_{2R}(0)}\int_{B_{2R}(0)} \frac{|I_{R}(u)(\xi)-I_{R}(u)(\eta)|^2}{|\eta^{-1}\circ\xi|^{Q+2s}}\md \xi \md \eta\\
    \nonumber
    \leq& 2\int_{G}\int_{G} \frac{|u(\xi)-u(\eta)|^2}{|\eta^{-1}\circ\xi|^{Q+2s}}\md \xi \md \eta+C_5 \|u\|^{2}_{L^2(B_{2R}(0))}\\
    \nonumber
    \leq& 2\int_{G}\int_{G} \frac{|u(\xi)-u(\eta)|^2}{|\eta^{-1}\circ\xi|^{Q+2s}}\md \xi \md \eta+C_6 \|u\|^{2}_{L^{\frac{2Q}{Q-2s}}(B_{2R}(0))}\\
    \nonumber
    \leq & C_7\int_{G}\int_{G} \frac{|u(\xi)-u(\eta)|^2}{|\eta^{-1}\circ\xi|^{Q+2s}}\md \xi \md \eta,
\end{align}
where  $C_5$, $C_6$ and $C_7$ are positive constants independent of $u$.
This proves \eqref{part3}.

By Rellich-Kondrachov  theorem for fractional Sobolev spaces (see \cite[Theorem 7]{gho}), one sees that $I_{R}$ is compact.
The proof of Lemma \ref{lemma 4.2} is thereby completed.
\end{proof}

Now
 we can  show  the embedding map (\ref{a4.1}) is compact.

\begin{lemma}\label{lm4.3}
Let $0<s<1$ and  $2\leq p<\frac{2Q}{Q-2s}$. The embedding map (\ref{a4.1}) is compact.
\end{lemma}
\begin{proof}
By H\"older's equality and fractional Sobolev inequalities \eqref{4.1}, we have, for $2< p<\frac{2Q}{Q-2s}$ and
$\theta=1-\frac{2Q-p(Q-2s)}{4s}$,
\begin{align*}
\int_{G}|u|^{p}U(z,w)^{\frac{2Q}{Q-2s}-p}\md z\md w\leq& \left(\int_{G}|u|^{\frac{2Q}{Q-2s}}\md z\md w\right)^{\theta}
\left(\int_{G}|u|^{2}U^{\frac{4s}{Q-2s}}\md z\md w\right)^{1-\theta}\\
\leq& C\left(\int_{G}u\mathcal{L}_{s}u\md z\md w\right)^{\theta\frac{Q}{Q-2s}}
\left(\int_{G}|u|^{2}U^{\frac{4s}{Q-2s}}\md z\md w\right)^{1-\theta}.
\end{align*}
So we  need only to show the case $p=2$.

By Lemma \ref{lemma 4.2},
the imbedding map $I_{R}: W^{s,2}(G)\hookrightarrow L^{2}(B_{2R}(0))$
is compact. On the other hand, we have, by H\"older's inequality and fractional Sobolev inequalities \eqref{4.1},
\begin{equation*}
\begin{split}
&\int_{G}|f-I_{R}(f)|^{2}U(z,w)^{\frac{4s}{Q-2s}}\md z\md w\\
\leq& \int_{G\setminus
B_{R}(0)}|f|^{2}U(z,t)^{\frac{4s}{Q-2s}}\md z\md w\\
\leq&\left(\int_{G\setminus
B_{R}(0)}|f|^{\frac{2Q}{Q-2s}}\md z\md w\right)^{\frac{Q-2s}{Q}}\left(\int_{G\setminus
B_{R}(0)}U(z,w)^{\frac{2Q}{Q-2s}}\md z\md w\right)^{1-\frac{Q-2s}{Q}}\\
\leq& C\int_G f \mathcal{L}_s f\md z\md w  \left(\int_{G\setminus
B_{R}(0)}U(z,w)^{\frac{2Q}{Q-2s}}\md z\md w\right)^{1-\frac{Q-2s}{Q}}.
\end{split}
\end{equation*}
By using  polar coordinates \eqref{polar}, we have
\begin{align*}
\int_{G\setminus
B_{R}(0)}U(z,w)^{\frac{2Q}{Q-2s}}\md z\md w
\leq& \int_{G\setminus B_{R}(0)}\frac{1}{|(z,w)|^{2Q}}\md z\md w\\
=&|\Sigma|\frac{1}{QR^{Q}}\rightarrow0,\;\;R\rightarrow\infty,
\end{align*}
where $|\Sigma|$ is the volume of $\Sigma$. Therefore,
 the embedding map (\ref{a4.1}) is compact because  it is a limit of
compact operators.
This completes the proof of Lemma \ref{lm4.3}.
\end{proof}

Since the embedding map $W^{s,2}(G)\hookrightarrow L^{2}(G,
U(z,w)^{\frac{4s}{Q-2s}}\md z\md w)$ is compact,  the spectrum of
\begin{eqnarray}\label{4.5}
\mathcal{L}_{s}v=\mu U^{\frac{4s}{Q-2s}}v,\;\;v\in
W^{s,2}(G),
\end{eqnarray}
 is discrete. Furthermore, by Lemma \ref{lm3.2},  we have  the following corollary:
\begin{corollary}\label{co4.2}
 Let $\mu_{j},\;j=1,2,\cdots$,  be the eigenvalues
of \eqref{4.5} given in increasing order. Then
 $\mu_{1}=N_{n,m,s}$ is simple with eigenfunction
$U$ where $U$ and $N_{n,m,s}$ are given by \eqref{1.7} and \eqref{N_nms}.
\end{corollary}

Define
\begin{equation}\label{U_}
   U_{\mu,\eta}(\xi)=\mu^{\frac{Q-2s}{2}}U(\delta_{\mu}(\eta^{-1}\circ\xi)),\;\;\;\;\;\eta\in G,\; \mu>0 .
\end{equation}
By \eqref{2.2} and \eqref{L_s varphi_-s rho}, we have
\begin{align}\label{equation L_s u=u}
    \mathcal{L}_s U_{\mu,\eta}=N_{n,m,s} U_{\mu,\eta}^{\frac{Q+2s}{Q-2s}},
\end{align}
where $N_{n,m,s}$ is given in \eqref{N_nms}.
If we set $\eta=(z',w')\in G$, then   we have
\begin{equation}\label{L_s partial U}
    \begin{split}
           &\mathcal{L}_s\frac{\partial  U_{\mu,\eta}}{\partial
  z'_{j}}=N_{n,m,s}\frac{Q+2s}{Q-2s} U_{\mu,\eta}^{\frac{4s}{Q-2s}}\frac{\partial  U_{\mu,\eta}}{\partial
  z'_{j}},\;\;\;\;j=1,,\cdots,2n;\\
 &\mathcal{L}_s\frac{\partial  U_{\mu,\eta}}{\partial
  w'_{r}}=N_{n,m,s}\frac{Q+2s}{Q-2s} U_{\mu,\eta}^{\frac{4s}{Q-2s}}\frac{\partial  U_{\mu,\eta}}{\partial
  w'_{r}},\;\;\;\;r=1,,\cdots,m;\\
  &\mathcal{L}_s\frac{\partial  U_{\mu,\eta}}{\partial
  \mu}=N_{n,m,s}\frac{Q+2s}{Q-2s} U_{\mu,\eta}^{\frac{4s}{Q-2s}}\frac{\partial  U_{\mu,\eta}}{\partial \mu}.
    \end{split}
\end{equation}
 Set
\begin{equation}\label{def of omega_i}
    \begin{split}
            \omega_{j}&=\frac{4}{Q-2s}U(\xi)^{-1}\left.\frac{\partial U_{\mu,\eta}}{\partial
  z'_{j}}\right|_{\mu=1,\eta=0},\;\;j=1,\cdots,2n;\\
     \omega_{2n+r}&=\frac{4}{Q-2s}\left.U(\xi)^{-1}\frac{\partial U_{\mu,\eta}}{\partial
  w'_{r}}\right|_{\mu=1,\eta=0},\;\;r=1,\cdots,m;  \\
\omega_{2n+m+1}&=\frac{2}{Q-2s}U(\xi)^{-1} \left. \frac{\partial U_{\mu,\eta}}{\partial
 \mu}\right|_{\mu=1,\eta=0},
    \end{split}
\end{equation}
Then we have (see \cite{ya1})
\begin{align}\label{sum omega_j}
   \sum_{j=1}^{2n+m+1} \omega_j^2=1
\end{align}
and by (\ref{L_s partial U}),
\begin{eqnarray}\label{4.12}
\mathcal{L}_{s}\Big(U\omega_{j}\Big)=N_{n,m,s}\frac{Q+2s}{Q-2s} U^{\frac{Q+2s}{Q-2s}}\omega_{j},\;\;j=1,2,\cdots,2n+m+1.
\end{eqnarray}

By Lemma \ref{lm4.3}, the
minimization problem
\begin{align}\label{4.6}
 \Lambda_{p}=\inf\left\{\int_{G}f \mathcal{L}_s f \md z\md w : \; \int_{G}|f|^{p}U(z,w)^{\frac{2Q}{Q-2s}-p}\md z\md w=1\right\}
\end{align}
for $2\leq p<\frac{2Q}{Q-2s}$ has a  solution $u$.
Since
\begin{align*}
\int_{G}|u| \mathcal{L}_s |u| \md \xi=&
a_{n,m,s}\int_G\int_G \frac{||u(\xi)|-|u(\eta)||^2}{|\eta^{-1}\circ\xi|^{Q+2s}}\md \xi \md \eta\\
\leq&a_{n,m,s}\int_G\int_G \frac{|u(\xi)-u(\eta)|^2}{|\eta^{-1}\circ\xi|^{Q+2s}}\md \xi \md \eta\\
=&\int_{G}u \mathcal{L}_s u \md\xi,
\end{align*}
we have that  if $u$ is a solution of \eqref{4.6}, so does $|u|$. So we may assume $u\geq0$.

Next we  show that such $u$ satisfies
a moment zero condition.
\begin{lemma}\label{lm3.3}
Let $0<s<1$,  $2\leq p<\frac{2Q}{Q-2s}$ and $u$ be a  nonnegative  solution of \eqref{4.6}.
Then we have
\begin{align}\label{3.2}
\int_{G}u^{p}U(z,w)^{\frac{2Q}{Q-2s}-p}\omega_{j}\md z\md w=0,\;j=1,2,\cdots,2n+m+1,
\end{align}
where $\omega_{j}\;(1\leq j\leq 2n+m+1)$ are  given by \eqref{def of omega_i}.
\end{lemma}
\begin{proof}
Set
\begin{align*}
\mathcal{F}_{p}(u)=&\frac{\int_{G}u \mathcal{L}_s u\md z\md w}{\left(\int_{G}u^{p}U(z,w)^{\frac{2Q}{Q-2s}-p}\md z\md w\right)^{\frac{2}{p}}};\\
u_{\mu^{-1},\eta^{-1}}(\xi)=&\mu^{-\frac{Q-2s}{2}}u(\delta_{\mu^{-1}}(\eta\circ\xi)),
\end{align*}
where $\mu>0$ and  $\eta=(z',w')\in G$.
By using \eqref{2.2}, we have
\begin{align*}
 \int_{G}u_{\mu^{-1},\eta^{-1}}\mathcal{L}_s u_{\mu^{-1},\eta^{-1}}\md z\md w=&\int_{G}u\mathcal{L}_s u\md z\md w;\\
 \int_{G}u_{\mu^{-1},\eta^{-1}}^{p}U(z,w)^{\frac{2Q}{Q-2s}-p}\md z\md w=&\int_{G}u^{p}U_{\mu,\eta}(z,w)^{\frac{2Q}{Q-2s}-p}\md z\md w,
\end{align*}
where $U_{\mu,\eta}$ is given by \eqref{U_}. Therefore,
\begin{align}\label{3.3}
\mathcal{F}_{p}(u_{\mu^{-1},\eta^{-1}})=
\frac{\int_{G}u\mathcal{L}_s u \md z\md w}{\left(\int_{G}u^{p}U_{\mu,\eta}(z,w)^{\frac{2Q}{Q-2s}-p}\md z\md w\right)^{\frac{2}{p}}}.
\end{align}
Since $u$ is a nonnegative  solution of \eqref{4.6}, we have
\begin{equation}\label{4.155}
  \begin{split}
    \frac{\partial}{\partial z'_{j}}\mathcal{F}_{p}(u_{\mu^{-1},\eta^{-1}})|_{\mu=1,\eta=0}=&0,\; j=1,\cdots,2n;\\
      \frac{\partial}{\partial w'_{r}}\mathcal{F}_{p}(u_{\mu^{-1},\eta^{-1}})|_{\mu=1,\eta=0}=&0,\; r=1,\cdots,m;\\
    \frac{\partial}{\partial \mu}\mathcal{F}_{p}(u_{\mu^{-1},\eta^{-1}})|_{\mu=1,\eta=0}=&0.
  \end{split}
\end{equation}
The desired result follows.
\end{proof}

\begin{lemma}\label{L_s f w_i, fw_i}
It holds that for any $f\in W^{s,2}(G)$,
\begin{align*}
\sum_{j=1}^{2n+m+1}\int_{G}f\omega_{j}\mathcal{L}_{s}( f \omega_j)\md \xi\leq \int_{G}f\mathcal{L}_{s}f\md \xi+\frac{4s}{Q-2s}N_{n,m,s}
\int_{G}\frac{f^{2}}{[(1+\frac{|z|^2}{4})^2+|w|^2]^s}\md \xi.
\end{align*}

\end{lemma}
\begin{proof}
By choosing $g=U$ and $h=f^2U^{-1}$ in \eqref{3.9}, we have
\begin{equation}\label{f^2 U^-1}
    \begin{split}
        \int_{G} f^{2}U^{-1} \mathcal{L}_{s} U \md\xi=&\int_{G} U \mathcal{L}_{s} (f^{2}U^{-1}) \md\xi\\
        =&\int_G f\mathcal{L}_s f\md \xi -a_{n,m,s}\int_{G}\int_{G}\left|\frac{f(\xi)}{U^2(\xi)}-\frac{f(\eta)}{U^2(\eta)}\right|^2\frac{U(\xi)
U(\eta)}{| \eta^{-1}\circ\xi|^{Q+2s}}\md\xi \md\eta,
    \end{split}
\end{equation}
Similarly, choosing $g=U\omega_j$ and $h=f^2\omega_jU^{-1} $ in \eqref{3.9} yields, for $j=1,2,\cdots,2n+m+1$,
\begin{equation}\label{f^2 omega_i U^{-1}}
    \begin{split}
        \int_{G} f^{2}\omega_jU^{-1}\mathcal{L}_{s} (U\omega_j) \md\xi=&   \int_{G}U\omega_j \mathcal{L}_{s} (f^{2}\omega_jU^{-1}) \md\xi\\
        =&\int_{G}f\omega_j\mathcal{L}_s (f\omega_j)\md\xi-a_{n,m,s}\times\\
           &\int_{G}\int_{G}\left|\frac{f(\xi)}{U^2(\xi)}-\frac{f(\eta)}{U^2(\eta)}\right|^2\frac{U(\xi)
U(\eta)}{| \eta^{-1}\circ\xi|^{Q+2s}}\omega_j(\xi)\omega_j(\eta)\md\xi \md\eta.
    \end{split}
\end{equation}
Combining \eqref{equation L_s u=u} and \eqref{f^2 U^-1} yields
\begin{align}\nonumber
    \int_G f\mathcal{L}_s f \md \xi&=N_{n,m,s}\int_G \frac{ f^2}{[(1+\frac{|z|^2}{4})^2+|w|^2]^s}\md \xi\\
    \label{b4.24}
    &\quad+a_{n,m,s}\int_{G}\int_{G}\left|\frac{f(\xi)}{U^2(\xi)}-\frac{f(\eta)}{U^2(\eta)}\right|^2\frac{U(\xi)
U(\eta)}{| \eta^{-1}\circ\xi|^{Q+2s}}\md\xi \md\eta.
\end{align}
On the other hand,
summing \eqref{f^2 omega_i U^{-1}} over $j$ and using \eqref{sum omega_j} and \eqref{4.12}, we obtain
\begin{align}\nonumber
   & \sum_{j=1}^{2n+m+1}\int_{G}f\omega_{j}\mathcal{L}_{s}( f \omega_j)\md \xi\\
   \nonumber
   =&N_{n,m,s}\frac{Q+2s}{Q-2s} \int_G \frac{f^2}{[(1+\frac{|z|^2}{4})^2+|w|^2]^s}\md z\md w+\\
   \nonumber
    &a_{n,m,s}\int_{G}\int_{G}\left|\frac{f(\xi)}{U^2(\xi)}-\frac{f(\eta)}{U^2(\eta)}\right|^2\frac{U(\xi)
U(\eta)}{| \eta^{-1}\circ\xi|^{Q+2s}}\sum_{j=1}^{2n+m+1}\omega_j(\xi)\omega_j(\eta)\md\xi \md\eta\\
\nonumber
\leq&N_{n,m,s}\frac{Q+2s}{Q-2s} \int_G \frac{f^2}{[(1+\frac{|z|^2}{4})^2+|w|^2]^s}\md z\md w+\\
\label{b4.25}
&a_{n,m,s}\int_{G}\int_{G}\left|\frac{f(\xi)}{U^2(\xi)}-\frac{f(\eta)}{U^2(\eta)}\right|^2\frac{U(\xi)
U(\eta)}{| \eta^{-1}\circ\xi|^{Q+2s}}\md\xi \md\eta.
\end{align}
To get \eqref{b4.25}, we use the  Cauchy-Schwarz inequality
\begin{align*}
    \left|\sum_{j=1}^{2n+m+1}\omega_j(\xi)\omega_j(\eta)\right|\leq \left(\sum_{j=1}^{2n+m+1}\omega_j^2(\xi)\right)^{1/2}\left(\sum_{j=1}^{2n+m+1}\omega_j^2(\eta)\right)^{1/2}=1.
\end{align*}
The desired result follows by combining \eqref{b4.24} and \eqref{b4.25}.
\end{proof}

Now we can give the proof of Theorem \ref{th1.2}. The idea is due to Frank and Lieb \cite{fl2,fl3} and Hang and Wang \cite{hang}.

\vspace{0.4cm}

\textbf{Proof of Theorem \ref{th1.2}}. Let $0<s<1$,  $2\leq p<\frac{2Q}{Q-2s}$ and $u_{p}$ be a  nonnegative  solution of \eqref{4.6}.
The 2nd variation of the functional $\mathcal{F}_{p}$ around $u_{p}$ shows that
\begin{align*}%\label{3.7}
\int_{G}f\mathcal{L}_s f\md z\md w&\int_{G}u_{p}^{p}U(z,w)^{\frac{2Q}{Q-2s}-p}\md z\md w-\\
&(p-1)
\int_{G}u_{p}\mathcal{L}_s u_p\md z\md w \int_{G}u_{p}^{p-2}U(z,w)^{\frac{2Q}{Q-2s}-p}f^{2}\md z\md w\geq0,
\end{align*}
for any $f$ with
\begin{align*}
\int_{G}u_{p}^{p-1}U(z,w)^{\frac{2Q}{Q-2s}-p}f\md z\md w=0.
\end{align*}
 Choosing  $f_j=u_{p}\omega_{j}$, $j=1,2,\cdots,2n+m+1$ and
summing over $j$, we obtain,
  in view of \eqref{sum omega_j} and Lemma \ref{L_s f w_i, fw_i},
\begin{align*}
(p-1)\int_G u_p\mathcal{L}_s u_p\md z\md w &\leq \sum_{j=1}^{2n+m+1}\int_G u_p\omega_j\mathcal{L}_s(u_p\omega_j)\md z\md w\\
&\leq \int_G u_p\mathcal{L}_s u_p\md z\md w+\frac{4s}{Q-2s}N_{n,m,s}\int_G \frac{u_p^2}{[(1+\frac{|z|^2}{4})^2+|w|^2]^s}\md z\md w.
\end{align*}
That is,
\begin{align*}
&(p-2)\left(\int_G u_p\mathcal{L}_s u_p\md z\md w-N_{n,m,s}\int_G \frac{u_p^2}{[(1+\frac{|z|^2}{4})^2+|w|^2]^s}\md z\md w\right)\\
\leq& N_{n,m,s}\left( \frac{2Q}{Q-2s}-p  \right)\int_G \frac{u_p^2}{[(1+\frac{|z|^2}{4})^2+|w|^2]^s} \md z\md w\\
\leq&N_{n,m,s}\left(\frac{2Q}{Q-2s}-p\right)\left(\int_{G}u_{p}^{p}U(z,w)^{\frac{2Q}{Q-2s}-p}\md z\md w\right)^{\frac{2}{p}}
\left(\int_{G}U(z,w)^{\frac{2Q}{Q-2s}}\md z\md w\right)^{1-\frac{2}{p}}\\
=&N_{n,m,s}\left(\frac{2Q}{Q-2s}-p\right)
\left(\int_{G}U(z,w)^{\frac{2Q}{Q-2s}}\md z\md w\right)^{1-\frac{2}{p}}\rightarrow0,\; p\nearrow \frac{2Q}{Q-2s}.
\end{align*}
To get the last equality, we use the fact that $\int_{G}u_{p}^{p}U(z,w)^{\frac{2Q}{Q-2s}-p}\md z\md w=1$.
Therefore, by Lemma \ref{lm3.2}, we get
\begin{align*}
\int_G u_p\mathcal{L}_s u_p\md z\md w-N_{n,m,s}\int_G \frac{u_p^2}{[(1+\frac{|z|^2}{4})^2+|w|^2]^s}\md z\md w\rightarrow0,\; p\nearrow \frac{2Q}{Q-2s}.
\end{align*}
Closely following the proof given in  \cite[Lemma 3.5]{ya1}, one can choose a sequence $\{p_{k}: k=1,2,\cdots \}$
such that $ p_{k}\nearrow \frac{2Q}{Q-2s}$ and   $u_{p_{k}}$ converges to a nonzero function $c_{0}U$.  Thus $c_{0}U$  is an  extremal function of
 \begin{align*}
 \Lambda_{p}=\inf\left\{\int_{G}f \mathcal{L}_s f \md z\md w : \; \int_{G}|f|^{\frac{2Q}{Q-2s}}\md z\md w=1\right\}.
 \end{align*}

Now we compute the sharp constant $S_{n,m,s}$. Since $U$  is an  extremal function, we have
\begin{align*}
    S_{n,m,s}=\frac{\int_G U\mathcal{L}_s U\md z\md w}{\left(\int_G U^{\frac{2Q}{Q-2s}}\md z\md w\right)^{(Q-2s)/Q}}
        =N_{n,m,s}\left(\int_G U(z,w)^{\frac{2Q}{Q-2s}}\md z\md w\right)^{2s/Q}.
\end{align*}
We compute
\begin{align}\nonumber
    \int_G U(z,w)^{\frac{2Q}{Q-2s}}\md z\md w =\int_GJ_{\mathscr{C}}(\xi)\md\xi=&\int_{\R^{2n}\times \R^{m}} \left(\left(1+\frac{|z|^2}{4} \right)^2 +|w|^2\right)^{-Q/2}\md z\md w\\
    \nonumber
    =&4^{n}\int_{\R^{2n}}(1+|z|^2)^{-(Q-m)}\md z \int_{\R^m}(1+|w|^2)^{-Q/2}\md w\\
    \label{b4.26}
    =&4^{n}\pi^{n+\frac{m}{2}}\frac{\Gamma(n+\frac{m}{2})}{\Gamma (2n+m)}.
\end{align}
To get the last equality, we use \eqref{b2.16}.
Therefore,
\begin{align}\label{4.27}
   S_{n,m,s}&=N_{n,m,s}\left(  \int_GJ_{\mathscr{C}}(\xi)\md\xi\right)^{2s/Q}\\
   \nonumber
   &=4^{s\frac{Q+2n}{Q}}\pi^{s\frac{2n+m}{Q}}\frac{\Gamma(\frac{n+1+s}{2})\Gamma(\frac{n+m+s}{2})}{\Gamma(\frac{n+1-s}{2})\Gamma(\frac{n+m-s}{2})}\left(  \frac{\Gamma(n+\frac{m}{2})}{\Gamma (2n+m)}\right)^{2s/Q}.
\end{align}
This completes the proof of  Theorem \ref{th1.2}.

\section{Proof of Theorem \ref{th1.3} and \ref{th1.4}}

We first give the proof of Theorem \ref{th1.3}.
\vspace{0.4cm}

\textbf{Proof of Theorem \ref{th1.3}}. By Corollary \ref{co1.1}, we have
\begin{align}\nonumber
 &\int_G\int_G \frac{|f(\xi)-f(\eta)|^2}{|\eta^{-1}\circ\xi|^{Q-2s}}J_{\mathscr{C}}^{\frac{Q+2s}{2Q}}(\xi)J_{\mathscr{C}}^{\frac{Q+2s}{2Q}}(\eta)\md\xi \md\eta \\
 \nonumber
 =&2 \int_G\int_G \frac{f^2(\xi)}{|\eta^{-1}\circ\xi|^{Q-2s}}J_{\mathscr{C}}^{\frac{Q+2s}{2Q}}(\xi)J_{\mathscr{C}}^{\frac{Q+2s}{2Q}}(\eta)\md\xi \md\eta
- 2\int_{G}\int_{G} \frac{f(\xi)f(\eta)}{|\eta^{-1}\circ\xi|^{Q-2s}}J_{\mathscr{C}}^{\frac{Q+2s}{2Q}}(\xi)J_{\mathscr{C}}^{\frac{Q+2s}{2Q}}(\eta)\md\xi \md\eta\\
\nonumber
\geq&2 \int_G\int_G \frac{f^2(\xi)}{|\eta^{-1}\circ\xi|^{Q-2s}}J_{\mathscr{C}}^{\frac{Q+2s}{2Q}}(\xi)J_{\mathscr{C}}^{\frac{Q+2s}{2Q}}(\eta)\md\xi \md\eta
-\frac{2}{S_{n,m,s}c_{n,m,s}}\left\|fJ_{\mathscr{C}}^{\frac{Q+2s}{2Q}}\right\|_{\frac{2Q}{Q+2s}}^2\\
\label{5.1}
=&\frac{2}{S_{n,m,s}c_{n,m,s}}\left(\frac{S_{n,m,s}}{N_{n,m,s}}\int_G f^2(\xi)J_{\mathscr{C}}(\xi)\md \xi-\left\|fJ_{\mathscr{C}}^{\frac{Q+2s}{2Q}}\right\|_{\frac{2Q}{Q+2s}}^2\right).
\end{align}
To get the last equality, we use the fact
\begin{align}\nonumber
    J_{\mathscr{C}}^{\frac{Q-2s}{2Q}}(\xi)=U(\xi)=\mathcal{L}_{s}^{-1}\mathcal{L}_{s}U(\xi)=&
    c_{n,m,s}N_{n,m,s}\int_{G}\frac{U^{\frac{Q+2s}{Q-2s}}(\eta)}{|\eta^{-1}\circ\xi|^{Q-2s}}\md \eta\\
    \label{5.2}
    =&
    c_{n,m,s}N_{n,m,s}\int_{G}\frac{J_{\mathscr{C}}^{\frac{Q+2s}{2Q}}(\eta)}{|\eta^{-1}\circ\xi|^{Q-2s}}\md \eta.
\end{align}
Passing to the limit as $s\searrow 0$ in \eqref{5.1} and using
\begin{align*}
     \lim_{s\searrow 0}\frac{S_{n,m,s}c_{n,m,s}}{s}=\frac{\Gamma(\frac{n+1}{2})\Gamma(\frac{n+m}{2})}{2^{n+1}\pi^{n+
\frac{m+1}{2}}},
\end{align*}
we get
\begin{align}\nonumber
&\int_G\int_G \frac{|f(\xi)-f(\eta)|^2}{|\eta^{-1}\circ\xi|^{Q-2s}}J_{\mathscr{C}}^{1/2}(\xi)J_{\mathscr{C}}^{1/2}(\eta)\md\xi \md\eta \\
\label{5.3}
\geq & \frac{2^{n+2}\pi^{n+
\frac{m+1}{2}}}{\Gamma(\frac{n+1}{2})\Gamma(\frac{n+m}{2})}
\lim_{s\searrow 0}\frac{1}{s}\left(\frac{S_{n,m,s}}{N_{n,m,s}}\int_G f^2(\xi)J_{\mathscr{C}}(\xi)\md \xi-\left\|fJ_{\mathscr{C}}^{\frac{Q+2s}{2Q}}\right\|_{\frac{2Q}{Q+2s}}^2\right).
\end{align}
For simplicity, we let
\begin{align*}
|\mathbb{S}^{2n+m}|=\int_{\mathbb{S}^{2n+m}}\md\sigma=\int_G J_{\mathscr{C}}(\xi)\md \xi.
\end{align*}
By using the  Cayley transform   (see Remark \ref{remark1.1}) and \eqref{4.27} and following the proof given in \cite[Corollary 2.5]{fl2},
we obtain
\begin{align}\nonumber
&\lim_{s\searrow 0}\frac{1}{s}\left(\frac{S_{n,m,s}}{N_{n,m,s}}\int_G f^2(\xi)J_{\mathscr{C}}(\xi)\md \xi-\left\|fJ_{\mathscr{C}}^{\frac{Q+2s}{2Q}}\right\|_{\frac{2Q}{Q+2s}}^2\right)\\
\nonumber
=&\lim_{s\searrow 0}\frac{1}{s}\left(|\mathbb{S}^{2n+m}|^{\frac{Q+2s}{Q}}-\|f\circ \mathscr{C}^{-1}\|^2_{L^{\frac{2Q}{Q+2s}}(\mathbb{S}^{2n+m})}\right)\\
\label{5.4}
=&\frac{2}{Q}\int_G f^2(\xi)\ln f^2(\xi) J_{\mathscr{C}}(\xi)\md \xi.
\end{align}
Substituting \eqref{5.4}  into \eqref{5.3}, we get
\begin{align*}
       \int_G\int_G \frac{|f(\xi) -f(\eta)|^2}{|\eta^{-1}\circ\xi|^Q}J^{1/2}_{\mathscr{C}}(\xi)J^{1/2}_{\mathscr{C}}(\eta)\md \xi \md \eta
     \geq \frac{2^{n+3}\pi^{n+\frac{m+1}{2}}}{Q\Gamma(\frac{n+1}{2})\Gamma(\frac{n+m}{2})}\int_G f^2(\xi)\ln f^2(\xi) J_{\mathscr{C}}(\xi)\md \xi.
\end{align*}

Next we show the sharpness. Following \cite{fl2},   we let
\begin{align*}
F_{\varepsilon}(\xi)=C_{\varepsilon}\left(\sqrt{1-\varepsilon^{2}}+\varepsilon \omega_{1}\right),\; 0<\varepsilon<\frac{1}{2},
\end{align*}
where $\omega_1$ is given in \eqref{def of omega_i} and $C_{\varepsilon}$ is such that $\int_{G}|F_{\varepsilon}(\xi)|^{2}J_{\mathscr{C}}(\xi)\md\xi=|\mathbb{S}^{2n+m}|$. A simple calculation shows
\begin{align}\label{5.5}
C_{\varepsilon}=|\mathbb{S}^{2n+m}|^{1/2}\left(\int_{G}|\sqrt{1-\varepsilon^{2}}+\varepsilon \omega_{1}|^{2}J_{\mathscr{C}}(\xi)\md\xi\right)^{-1/2}\rightarrow 1,\;\varepsilon\searrow 0.
\end{align}
We compute
\begin{align}\nonumber
&\int_G\int_G \frac{|F_{\varepsilon}(\xi)-F_{\varepsilon}(\eta)|^2}{|\eta^{-1}\circ\xi|^{Q-2s}}J_{\mathscr{C}}^{\frac{Q+2s}{2Q}}(\xi)J_{\mathscr{C}}^{\frac{Q+2s}{2Q}}(\eta)\md\xi \md\eta \\
\nonumber
=&C_{\varepsilon}^{2}\varepsilon^{2}\int_G\int_G \frac{|\omega_{1}(\xi)-\omega_1(\eta)|^2}{|\eta^{-1}\circ\xi|^{Q-2s}}J_{\mathscr{C}}^{\frac{Q+2s}{2Q}}(\xi)J_{\mathscr{C}}^{\frac{Q+2s}{2Q}}(\eta)\md\xi \md\eta \\
 \nonumber
 =&2C_{\varepsilon}^{2}\varepsilon^{2} \int_G\int_G \frac{\omega_{1}^{2}(\xi)}{|\eta^{-1}\circ\xi|^{Q-2s}}J_{\mathscr{C}}^{\frac{Q+2s}{2Q}}(\xi)J_{\mathscr{C}}^{\frac{Q+2s}{2Q}}(\eta)\md\xi \md\eta-\\
 \label{5.6}
& 2C_{\varepsilon}^{2}\varepsilon^{2}\int_{G}\int_{G} \frac{\omega_{1}(\xi)\omega_{1}(\eta)}{|\eta^{-1}\circ\xi|^{Q-2s}}J_{\mathscr{C}}^{\frac{Q+2s}{2Q}}(\xi)J_{\mathscr{C}}^{\frac{Q+2s}{2Q}}(\eta)\md\xi \md\eta.
\end{align}
By \eqref{4.12} and Lemma \ref{lm2.2}, we have
\begin{align}\nonumber
    J_{\mathscr{C}}^{\frac{Q-2s}{2Q}}(\xi)\omega_1(\xi)=&\mathcal{L}_{s}^{-1}\mathcal{L}_{s}\Big(U(\xi)\omega_1(\xi)\Big)\\
    \label{5.7}
  =  &
    c_{n,m,s}N_{n,m,s}\frac{Q+2s}{Q-2s}\int_{G}\frac{J_{\mathscr{C}}^{\frac{Q+2s}{2Q}}(\eta)\omega_{1}(\eta)}{|\eta^{-1}\circ\xi|^{Q-2s}}\md \eta.
\end{align}
Substituting \eqref{5.2} and \eqref{5.7} into \eqref{5.6}, we obtain
\begin{align*}
&\int_G\int_G \frac{|F_{\varepsilon}(\xi)-F_{\varepsilon}(\eta)|^2}{|\eta^{-1}\circ\xi|^{Q-2s}}J_{\mathscr{C}}^{\frac{Q+2s}{2Q}}(\xi)J_{\mathscr{C}}^{\frac{Q+2s}{2Q}}(\eta)\md\xi  \\\\
=&\frac{8C_{\varepsilon}^{2}\varepsilon^{2}s}{c_{n,m,s}N_{n,m,s}(Q+2s)}\int_{G}\omega_1^{2}(\xi)J_{\mathscr{C}}(\xi)\md\xi.
\end{align*}
Therefore,
\begin{align}\nonumber
&\int_G\int_G \frac{|F_{\varepsilon}(\xi)-F_{\varepsilon}(\eta)|^2}{|\eta^{-1}\circ\xi|^{Q}}J_{\mathscr{C}}^{1/2}(\xi)J_{\mathscr{C}}^{1/2}(\eta)\md\xi  \\
\nonumber
=&\lim_{s\searrow 0}\frac{8C_{\varepsilon}^{2}\varepsilon^{2}s}{c_{n,m,s}N_{n,m,s}(Q+2s)}\int_{G}\omega_1^{2}(\xi)J_{\mathscr{C}}(\xi)\md\xi\\
\label{5.8}
=&\frac{2^{n+4}\pi^{n+\frac{m+1}{2}}}{Q\Gamma(\frac{n+1}{2})\Gamma(\frac{n+m}{2})}\varepsilon^{2}\int_{G}\omega_1^{2}(\xi)J_{\mathscr{C}}(\xi)\md\xi.
\end{align}

On the other hand,
\begin{align}\nonumber
&\int_G F_{\varepsilon}^{2}(\xi)\ln F_{\varepsilon}^{2}(\xi) J_{\mathscr{C}}(\xi)\md \xi\\
\nonumber
=&C_{\varepsilon}^{2}
\int_G \left(1+2\varepsilon \sqrt{1-\varepsilon^{2}}\omega_{1}-\varepsilon^{2}(1-\omega_{1}^{2})\right)
\left[\ln\left(1+2\varepsilon \sqrt{1-\varepsilon^{2}}\omega_{1}-\varepsilon^{2}(1-\omega_{1}^{2})\right)\right.\\
\label{5.9}
&+\left.\ln\frac{|\mathbb{S}^{2n+m}|}{\int_{G}|\sqrt{1-\varepsilon^{2}}+\varepsilon \omega_{1}|^{2}J_{\mathscr{C}}(\xi)\md\xi}\right] J_{\mathscr{C}}(\xi)\md \xi.
\end{align}
It is easy to verify that
\begin{equation}\label{5.10}
  \begin{split}
1+2\varepsilon \sqrt{1-\varepsilon^{2}}\omega_{1}-\varepsilon^{2}(1-\omega_{1}^{2})=&1+2\varepsilon \omega_{1}+O(\varepsilon^{2});\\
\ln\left(1+2\varepsilon \sqrt{1-\varepsilon^{2}}\omega_{1}-\varepsilon^{2}(1-\omega_{1}^{2})\right)=&
2\varepsilon \omega_{1}-\varepsilon^{2}(1+\omega_{1}^{2})+O(\varepsilon^{3}).
  \end{split}
\end{equation}

Noticing that $U(\xi)$ and $U(\xi)\omega_1(\xi)$ are eigenfunctions corresponding to different eigenvalues of \eqref{4.5},
we have
\begin{align}\label{5.11}
\int_{G}\omega_1(\xi)J_{\mathscr{C}}(\xi)\md\xi=\int_{G}\omega_1(\xi)U(\xi)\cdot U(\xi)\cdot U^{\frac{4s}{Q-2s}}(\xi)\md\xi=0.
\end{align}
By using \eqref{5.11}, we have
\begin{align}\nonumber
\ln\frac{|\mathbb{S}^{2n+m}|}{\int_{G}|\sqrt{1-\varepsilon^{2}}+\varepsilon \omega_{1}|^{2}J_{\mathscr{C}}(\xi)\md\xi}=&\nonumber
-\ln\frac{\int_{G}(1-\varepsilon^{2}+\varepsilon^{2}\omega_{1}^{2})J_{\mathscr{C}}(\xi)\md\xi}{|\mathbb{S}^{2n+m}|}\\
\label{5.12}
=&
\varepsilon^{2}
\left(1-\frac{1}{|\mathbb{S}|^{2n+m}}\int_{G}\omega_1^{2}(\xi)J_{\mathscr{C}}(\xi)\md\xi\right)+O(\varepsilon^{3}).
\end{align}

Substituting \eqref{5.10} and \eqref{5.12} into \eqref{5.9} and using \eqref{5.11}, we obtain
\begin{align}\label{5.13}
\int_G F_{\varepsilon}^{2}(\xi)\ln F_{\varepsilon}^{2}(\xi) J_{\mathscr{C}}(\xi)\md \xi=2\varepsilon^{2}C_{\varepsilon}^{2}
\int_{G}\omega_1^{2}(\xi)J_{\mathscr{C}}(\xi)\md\xi+O(\varepsilon^{3}).
\end{align}
By using \eqref{5.5},  \eqref{5.8} and \eqref{5.13}, we get
\begin{align*}
\frac{\int_G\int_G \frac{|F_{\varepsilon}(\xi)-F_{\varepsilon}(\eta)|^2}{|\eta^{-1}\circ\xi|^{Q}}J_{\mathscr{C}}^{1/2}(\xi)J_{\mathscr{C}}^{1/2}(\eta)\md\xi\md \eta  }
{\int_G F_{\varepsilon}^{2}(\xi)\ln F_{\varepsilon}^{2}(\xi) J_{\mathscr{C}}(\xi)\md \xi}
\rightarrow \frac{2^{n+3}\pi^{n+\frac{m+1}{2}}}{Q\Gamma(\frac{n+1}{2})\Gamma(\frac{n+m}{2})},\;\; \varepsilon\searrow 0.
\end{align*}
This proves that the constant $\frac{2^{n+3}\pi^{n+\frac{m+1}{2}}}{Q\Gamma(\frac{n+1}{2})\Gamma(\frac{n+m}{2})}$ is  sharp. The proof of Theorem \ref{th1.3}
is thereby completed.

\vspace{0.4cm}

Before the proof of Theorem \ref{th1.4}, we first recall the following extension problem on $G$:
\begin{align}\label{5.14}
    \begin{cases}
    \partial_{\rho\rho}u+\frac{1-2s}{\rho}\partial_{\rho}u+\frac{1}{4}\rho^2\Delta_w u-\mathcal{L} u=0, &\text{in}~\Omega= G\times \R^+,\\
    \lim\limits_{\rho \searrow0}u(z,w,\rho)=f(z,w)\in W^{s,2}(G), &\text{on}~\partial\Omega\simeq G,
    \end{cases}
\end{align}
where $0<s<1$. Let  $1<p<\infty$.
It has been shown by Roncal and Thangavelu \cite{ron2} (see   also
 Garofalo and Tralli \cite{ga1,ga2}    for  a  complete self-contained approach  to these construction via  heat equation methods) that if
\begin{align*}
\int_{G}|u(z,w,\rho)|^{p}\md z\md w\leq C
\end{align*}
uniformly in $\rho$, then the solution of \eqref{5.14}
is given by the  Poisson integral
\begin{align}\label{5.15}
\mathscr{P}(f)=C_1(n,m,s)\rho^{2s} f\ast \varphi_{s,\frac{1}{4}\rho^{2}},
\end{align}
Moreover,
if $\mathcal{L}_{s}f\in L^{p}(G)$,  then we have
\begin{align}\label{a5.1}
-\lim_{\rho\rightarrow0+}\rho^{1-2s}\partial_{\rho}\mathscr{P}(f)=2^{1-2s}\frac{\Gamma(1-s)}{\Gamma(s)}\mathcal{L}_{s}f,
\end{align}
where $C_1(n,m,s)=2^{-2n-2s}\pi^{-n-\frac{m}{2}}\frac{\Gamma(n+s)\Gamma(\frac{n+m+s}{2})}{\Gamma(s)\Gamma(\frac{n+s}{2})}$.

Now we can prove Theorem \ref{th1.4}.
\vspace{0.3cm}

\textbf{Proof of Theorem \ref{th1.4}.}
We first show \eqref{1.14}.
Without loss of generality, we assume $u\in C_{0}^{\infty}(G\times \mathbb{R})$ and let $f(z,w)=u(z,w,0)$.
For simplicity, we set
 \begin{align*}
  \nabla u=&\left(\partial_{\rho} u, \frac{1}{2}\rho\nabla_{w}u, \nabla_{G}\right),\\
  \Delta u=& \partial_{\rho\rho}u+\frac{1-2s}{\rho}\partial_{\rho}u+\frac{1}{4}\rho^2\Delta_w u-\mathcal{L} u,
 \end{align*}
 such that,  for $g,h\in C_{0}^{\infty}(G\times (0,\infty))$, it holds
 \begin{align*}
&\int_0^\infty\int_G \langle \nabla g(z,w,\rho),\nabla h(z,w,\rho)\rangle \cdot \rho^{1-2s}\md z\md w\md\rho\\
=&-\int_0^\infty\int_G g(z,w,\rho)\Delta h(z,w,\rho)  \rho^{1-2s}\md z\md w\md\rho.
 \end{align*}
Obviously,
\begin{align}\label{5.17}
\Delta \mathscr{P}(f)=0,\;\; (u-\mathscr{P}(f))|_{\rho=0}=0,
\end{align}
where $\mathscr{P}(f)$ is the  Poisson integral given by \eqref{5.15}.

We claim that
   \begin{align} \nonumber
        &\int_0^\infty\int_G |\nabla u|^2\rho^{1-2s}\md z\md w\md \rho\\
        \label{5.18}
        =&2^{1-2s}\frac{\Gamma(1-s)}{\Gamma(s)}\int_G f\mathcal{L}_{s}f\md \xi+
         \int_0^\infty\int_G\left|\nabla (u-\mathscr{P}(f))\right|^2 \rho^{1-2s}\md z\md w\md \rho.
   \end{align}
 In fact, through integrating by parts, we obtain, by using \eqref{5.17}
    \begin{align}\nonumber
        &\int_0^\infty\int_G\left|\nabla u-\nabla \mathscr{P}(f)\right|^2 \rho^{1-2s}\md z\md w\md \rho\\
        \nonumber
        =&\int_0^\infty\int_G |\nabla u|^2\rho^{1-2s}\md z\md w\md \rho-\int_0^\infty\int_G |\nabla \mathscr{P}(f)|^2\rho^{1-2s} \md z\md w\md \rho\\
        \nonumber
        &-2\int_0^\infty\int_G \langle\nabla (u-\mathscr{P}(f)),\nabla \mathscr{P}(f)\rangle \rho^{1-2s}\md z\md w\md \rho\\
        \nonumber
    =&\int_0^\infty\int_G |\nabla u(z,t,\rho)|^2\rho^{1-2s}\md z\md w\md \rho-\int_0^\infty\int_G |\nabla \mathscr{P}(f)|^2\rho^{1-2s} \md z\md w\md \rho \\
    \nonumber
    &\quad +2\int_0^\infty\int_G (u-\mathscr{P}(f))\Delta \mathscr{P}(f)\rho^{1-2s}\md z\md w\md \rho\\
    \label{5.19}
        =&\int_0^\infty\int_G |\nabla u|^2\rho^{1-2s}\md z\md w\md \rho-\int_0^\infty\int_G |\nabla \mathscr{P}(f)|^2\rho^{1-2s} \md z\md w\md \rho.
    \end{align}
On the other hand, by using  (\ref{5.17}) and (\ref{a5.1}), we have
\begin{align}\nonumber
\int_0^\infty\int_G\left|\nabla \mathscr{P}(f)\right|^2 \rho^{1-2s}\md z\md w\md \rho
=&-\int_G  f\lim_{\rho\rightarrow0+}\rho^{1-2s}\partial_{\rho}\mathscr{P}(f) \md z\md w\\
\label{5.16}
=&2^{1-2s}\frac{\Gamma(1-s)}{\Gamma(s)}\int_G f\mathcal{L}_{s}f\md z\md w.
\end{align}
Substituting \eqref{5.16} into \eqref{5.19}, we obtain  \eqref{5.18}. This proves the claim.

By using \eqref{5.18} and Theorem \ref{th1.2}, we get
 \begin{align*}
\int_0^\infty\int_G |\nabla u|^2\rho^{1-2s}\md z\md w\md \rho\geq &2^{1-2s}\frac{\Gamma(1-s)}{\Gamma(s)}\int_G f\mathcal{L}_{s}f\md \xi\\
\geq&2^{1-2s}\frac{\Gamma(1-s)}{\Gamma(s)} S_{n,m,s}\left(\int_{G}|f(\xi)|^{\frac{2Q}{Q-2s}}\md\xi\right)^{\frac{Q-2s}{Q}}.
 \end{align*}
 This proves \eqref{1.14}.

On the other hand, by \eqref{5.16} and Theorem \ref{th1.2}, one sees an extremal function of \eqref{1.14} is given
by
\begin{align*}
u=\mathscr{P}(U)=C_1(n,m,s)\rho^{2s} U\ast \varphi_{s,\frac{1}{4}\rho^{2}}.
\end{align*}
The proof of Theorem \ref{th1.4}
is thereby completed.

\end{document}